\documentclass[11pt]{amsart}

\usepackage{amssymb}
\usepackage{url}

\usepackage{wasysym}
\usepackage{amsmath}
\usepackage{mathrsfs}
\usepackage[pagebackref, hypertexnames=false, colorlinks, citecolor=red, linkcolor=red]{hyperref}

	\normalbaselines						  %

\def\RR{\mathbb{R}}

\def\CC{\mathbb{C}}
\def\NN{\mathbb{N}}

\newcommand{\al}{{\alpha}}

\newcommand{\f}{{\varphi}}

\newcommand{\cV}{{\mathcal{V}}}
\newcommand{\cX}{{\mathcal{X}}}

\newcommand{\R}{{\mathbb  R}}

\newcommand{\N}{{\mathbb  N}}

\newcommand{\fdot}{\,\cdot\,}

\makeatletter
\def\Ddots{\mathinner{\mkern1mu\raise\p@
\vbox{\kern7\p@\hbox{.}}\mkern2mu
\raise4\p@\hbox{.}\mkern2mu\raise7\p@\hbox{.}\mkern1mu}}
\makeatother

\newcommand{\cH}{\mathcal{H}}
\newcommand{\cC}{\mathcal{C}}
\newcommand{\cB}{\mathcal{B}}
\newcommand{\cF}{\mathcal{F}}
\newcommand{\cA}{\mathcal{A}}

\newcommand{\cD}{\mathcal{D}}

\DeclareMathOperator{\spa}{span}

\newcommand{\ci}[1]{_{ {}_{\scriptstyle #1}}}
\newcommand{\ti}[1]{_{\scriptstyle \text{\rm #1}}}

\catcode`@=11
\chardef\mathlig@atcode\count255

\def\actively#1#2{\begingroup\uccode`\~=`#2\relax\uppercase{\endgroup#1~}}
\def\mathlig@gobble{\afterassignment\mathlig@next@cmd\let\mathlig@next= }
\def\mathlig@delim{\mathlig@delim}
\def\mathlig@defcs#1{\expandafter\def\csname#1\endcsname}
\def\mathlig@let@cs#1#2{\expandafter\let\expandafter#1\csname#2\endcsname}
\def\mathlig@appendcs#1#2{\expandafter\edef\csname#1\endcsname{\csname#1\endcsname#2}}

\def\mathlig#1#2{\mathlig@checklig#1\mathlig@end\mathlig@defcs{mathlig@back@#1}{#2}\ignorespaces}

\def\mathlig@checklig#1#2\mathlig@end{%
 \expandafter\ifx\csname mathlig@forw@#1\endcsname\relax
 \expandafter\mathchardef\csname mathlig@back@#1\endcsname=\mathcode`#1%
 \mathcode`#1"8000\actively\def#1{\csname mathlig@look@#1\endcsname}%
 \mathlig@dolig#1\mathlig@delim
\fi
\mathlig@checksuffix#1#2\mathlig@end
}

\def\mathlig@checksuffix#1#2\mathlig@end{%
\ifx\mathlig@delim#2\mathlig@delim\relax\else\mathlig@checksuffix@{#1}#2\mathlig@end\fi
}
\def\mathlig@checksuffix@#1#2#3\mathlig@end{%
\expandafter\ifx\csname mathlig@forw@#1#2\endcsname\relax\mathlig@dosuffix{#1}{#2}\fi
\mathlig@checksuffix{#1#2}#3\mathlig@end
}

\def\mathlig@dosuffix#1#2{%
\mathlig@appendcs{mathlig@toks@#1}{#2}%
\mathlig@dolig{#1}{#2}\mathlig@delim
}

\def\mathlig@dolig#1#2\mathlig@delim{%
 \mathlig@defcs{mathlig@look@#1#2}{%
 \mathlig@let@cs\mathlig@next{mathlig@forw@#1#2}\futurelet\mathlig@next@tok\mathlig@next}%
 \mathlig@defcs{mathlig@forw@#1#2}{%
  \mathlig@let@cs\mathlig@next{mathlig@back@#1#2}%
  \mathlig@let@cs\checker{mathlig@chck@#1#2}%
  \mathlig@let@cs\mathligtoks{mathlig@toks@#1#2}%
  \expandafter\ifx\expandafter\mathlig@delim\mathligtoks\mathlig@delim\relax\else
  \expandafter\checker\mathligtoks\mathlig@delim\fi
  \mathlig@next
 }%
 \mathlig@defcs{mathlig@toks@#1#2}{}%
 \mathlig@defcs{mathlig@chck@#1#2}##1##2\mathlig@delim{%
  \ifx\mathlig@next@tok##1%
   \mathlig@let@cs\mathlig@next@cmd{mathlig@look@#1#2##1}\let\mathlig@next\mathlig@gobble
  \fi
  \ifx\mathlig@delim##2\mathlig@delim\relax\else
   \csname mathlig@chck@#1#2\endcsname##2\mathlig@delim
  \fi
 }%
 \ifx\mathlig@delim#2\mathlig@delim\else
  \mathlig@defcs{mathlig@back@#1#2}{\csname mathlig@back@#1\endcsname #2}%
 \fi
}%

\catcode`@\mathlig@atcode

\mathchardef\ordinarycolon\mathcode`\:
\def\vcentcolon{\mathrel{\mathop\ordinarycolon}}
\mathlig{:=}{\vcentcolon=}
\mathlig{::=}{\vcentcolon\vcentcolon=}

\numberwithin{equation}{section}

\theoremstyle{plain}
\newtheorem{theo}{Theorem}[section]
\newtheorem{cor}[theo]{Corollary}
\newtheorem{lem}[theo]{Lemma}
\newtheorem{prop}[theo]{Proposition}

\newtheorem{conj}[theo]{Conjecture}

\theoremstyle{definition}
\newtheorem{defn}[theo]{Definition}

\newtheorem*{theorem*}{Theorem}
\theoremstyle{remark}

\newtheorem*{ex*}{Example}
\theoremstyle{remark}
\newtheorem*{exs*}{Examples}
\theoremstyle{remark}
\newtheorem*{rems*}{Remarks}
\theoremstyle{remark}
\newtheorem*{rem*}{Remark}
\newtheorem{rem}[theo]{Remark}

\title[Boundary Conditions for Left-Definite Spaces]{Boundary Conditions associated with the General Left-Definite Theory for Differential Operators}

\author{Matthew~Fleeman}
\address{Department of Mathematics, Baylor University, One Bear Place \#97328,     
 Waco, TX  76798, USA}
\email{Matthew$\underline{\,\,\,}$Fleeman@baylor.edu}

\author{Dale~Frymark}
\address{Department of Mathematics, Baylor University, One Bear Place \#97328,     
 Waco, TX  76798, USA}
\email{Dale$\underline{\,\,\,}$Frymark@baylor.edu}

\author{Constanze~Liaw}
\address{Department of Mathematical Sciences, University of Delware, 501 Ewing Hall, Newark, DE  19716, USA}

\thanks{
The work of Constanze Liaw was supported by Simons Foundation Grant \#426258.}


\begin{document}

\begin{abstract}
In the early 2000's, Littlejohn and Wellman developed a general left-definite theory for certain self-adjoint operators by fully determining their domains and spectral properties. The description of these domains do not feature explicit boundary conditions. We present characterizations of these domains given by the left-definite theory for all operators which possess a complete system of orthogonal eigenfunctions, in terms of classical boundary conditions.
\end{abstract}

\maketitle

\setcounter{tocdepth}{1}
\tableofcontents

\section{Introduction}\label{s-into}

A wide variety of literature concerns the study of left-definite theory applied to Sturm--Liouville differential operators. The interest arises primarily from the groundbreaking paper of Littlejohn and Wellman \cite{LW02}. The paper describes the creation of a continuum of left-definite spaces and left-definite operators associated with an arbitrary self-adjoint operator that is bounded below by a positive constant in a Hilbert space. Prior to \cite{LW02}, research had been conducted only in the ``first" left-definite setting. This theory has been applied to many types of self-adjoint differential operators, including those stemming from the second-order differential equations of Hermite, Legendre, Jacobi, Laguerre, and Fourier. Excellent surveys of these results are \cite{BLTW} and \cite{LW13}.

The paper of Littlejohn and Wellman \cite{LW02} managed to characterize these left-definite spaces in terms of other Hilbert spaces defined with integral operators. A key point of left-definite theory is that each left-definite space will be nested and dense within the original Hilbert space. However, some critics felt somewhat uneasy with the fact that the left-definite spaces (in their opinions) \emph{lack} an explicit mention of boundary conditions.

Here, we present classical boundary conditions that exist for self-adjoint differential operators which possess a complete system of orthogonal eigenfunctions. These boundary conditions are formulated in terms of Glazman--Krein--Naimark (from now on abbreviated by GKN) theory, which entirely describes self-adjoint extensions for a closed, symmetric operator with equal deficiency indices. The formulation in terms of GKN conditions can be seen as a vast improvement over the previous description of these spaces, as verifying against a few functions is easier than proving certain boundedness properties. Progress has also been made towards showing different types of boundary conditions. In particular, the classical Sturm--Liouville differential operators are defined on $L^2[(a,b),w]$ spaces. The aforementioned GKN conditions are shown to be equivalent to the vanishing of certain limits of functions as they approach the endpoints $a$ and $b$. This provides an alternative form of testing whether an arbitrary function from the original Hilbert space also belongs to a certain left-definite space.

This framework allows for the construction of the boundary conditions for the left-definite spaces of the classical Legendre differential operator explicitly. Previous work using GKN conditions to describe the left-definite domains is limited to very recent progress by Littlejohn and Wicks \cite{LW15, LWOG}. The results concern the classical Legendre differential operator exclusively and give GKN conditions describing the fourth left-definite domain, which is associated with the square of the differential operator, ${\bf L}^2$. Additionally, Littlejohn and Wicks formulate their results in terms of ``separated" boundary conditions, whereas ``coupled" boundary conditions are used throughout this paper. This is a matter of preference, but using coupled boundary conditions simplifies calculations considerably, as they are easier to access via the sesquilinear form dealt with by GKN theory.

In this work we introduce a systematic approach, which reduces the amount of cumbersome computations in this field. This perspective enables us to harvest the finite dimensional nature of defect spaces. An alternative approach to this problem is through Sturm--Liouville theory, e.g.~\cite{FW, KWZ}, but the literature focuses on the first left-definite theory and does not produce GKN conditions.

The interest in differential operators which possess a complete system of orthogonal eigenfunctions originates with a result in \cite{LW02} that says this same system will be present in each of the different left-definite domains. Hence, there is an indicator for when a self-adjoint extension is a left-definite domain, and this simplifies the process. A second-order linear differential equation satisfied by a complete orthogonal system of polynomials with absolutely continuous measures of orthogonality has a second linearly independent solution \cite[Section 3.6]{I}. This second linearly independent solution is often called a function of the second kind, and their existence plays an essential role in our examples.

The Legendre differential operator example is particularly important because there are essentially only four Sturm--Liouville operators with a complete set of orthogonal eigenfunctions. The Bochner classification \cite{Bochner} tells us that, up to a complex linear change of variable, the only such operators with polynomial eigenfunctions are Jacobi, Hermite, Laguerre and Bessel. Of these, the Jacobi differential expressions require the most boundary conditions. The Legendre expression is a special case of Jacobi that has an immense amount of literature, so it was ideal for such an exploration of GKN conditions with respect to left-definite theory. The framework of Section \ref{s-legendre} both extends to other Jacobi differential expressions (other values of the parameters $\al$ and $\beta$), and reduces to cover the cases of Hermite, Laguerre and Bessel. The broader concepts of Section \ref{s-conj} are expressed with this in mind as well.

This general approach suggests other open problems: What adjustments are necessary when the self-adjoint operator has a spectrum that is not discrete? That is, given some non-standard extension of a differential operator, can we describe the corresponding left-definite domain? How about the left-definite domain for compositions of the operator?

\subsection{Outline}
Section \ref{s-background} deals primarily with two different areas: left-definite theory and self-adjoint extension theory. The left-definite theory mainly follows the classical results contained in \cite{LW02}, and includes the structure of left-definite spaces as well as key facts about their spectra. The self-adjoint extension theory follows the classical text of Naimark \cite{N}, and culminates in GKN theory. This theory is the foundation for the discussion of boundary conditions of self-adjoint operators. Subsection \ref{s-graph} describes the graph norm that can be endowed on a Hilbert space and justifies intuition about the decomposition of the maximal domain.

In Section \ref{s-framework} we construct a systematic framework around the method of finding GKN conditions that make differential expressions into self-adjoint differential operators. We show that eigenfunctions of a self-adjoint operator, which are linearly independent modulo the minimal domain, yield GKN conditions for the operator.

In Section \ref{s-legendre}, explicit self-adjoint extensions are given by showing that eigenfunctions themselves lead to suitable GKN conditions for powers ${\bf L}^n$ of the classical Legendre operator ${\bf L}$. The case $n=2$ was the topic of \cite{LW15}. One of the key features of our approach is that it utilizes the functions of the second kind. In examples, we prove that the first $n$ eigenfunctions work for ${\bf L}^n$ for $n=2, \hdots, 5$. The statement has been verified numerically for $n\le16$.
We show a necessary condition for eigenfunctions and functions of the second kind to be suitable ``test" functions for linear independence modulo the minimal domain. The sufficiency of this condition is discussed and conjectured.

Motivated by the Legendre example, it is shown in Section \ref{s-results} that for left-definite operators with pure point spectrum (only eigenvalues) there exist eigenfunctions (corresponding to some eigenvalues) that generate GKN conditions.
The method of proof for this result differs from the explicit computations that were used in the Legendre example. It relies on working with the graph norm. We believe this to be the first general result in this direction. The idea is to reduce the problem to its essence: finite dimensional linear algebra. This is accomplished by using the fact that the defect spaces and minimal domain are orthogonal with respect to graph norm, and properties of the eigenfunctions.

In Section \ref{s-conj} the goal is to determine which boundary conditions describe left-definite domains. This question is then explored by comparing the left-definite domains for these differential operators with the complete system of orthogonal polynomials and their GKN conditions to boundary conditions that stem from the definition of the Sturm--Liouville operator. The work improves and simplifies a proof of a fact from \cite{LWOG}. The central conjecture stating the equivalence of the following four sets is still partially open:
\begin{itemize}
\item The $n${th} left definite domain.
\item The maximal domain with GKN conditions determined by the first $m$ orthogonal polynomials. (Here, $m$ denotes the deficiency indices.)
\item The maximal domain with GKN conditions determined by any $m$ orthogonal polynomials.
\item The maximal domain with certain explicit boundary conditions.
\end{itemize}

Other ramifications and specific discussions of individual systems of orthogonal polynomials and their differential equations follows.

\subsection{Notation}\label{ss-notation}
We use $\ell$ to denote differential expressions (on a separable Hilbert space $\cH$). We mostly work with general Sturm--Liouville expressions in symmetric form. In the example we focus on the Legendre expression. Sets and spaces are generally denoted with ``mathcal"; the Hilbert space $\cH$, the minimal domain $\cD\ti{min}$, the defect spaces $\cD_+$ and $\cD_-$, etc.

Further, we use the notation $\{\ell, \cX\}$ to refer to an operator associated to $\ell$ with domain $\cX$. Since we work with unbounded operators, the operators are defined on dense subspaces $\cX\subsetneq \cH$. The maximal domain is denoted by $\cD\ti{max}$ (the largest subset of $\cH$ with $\ell (\cD\ti{max})\subset \cH$). Sometimes, $\cD\ti{max}(\ell)$ is used to emphasize the expression. Boldface letters are used for operators and matrices. We abbreviate the maximal operator by ${\bf L}\ti{max}$, i.e.~${\bf L}\ti{max}=\{\ell, \cD\ti{max}\}$. In analogy, ${\bf L}\ti{min}=\{\ell, \cD\ti{min}\}$ is the minimal operator, and in this context ${\bf L}=\{\ell, \cD_{\bf L}\}$ is used to denote self-adjoint operators.

When we consider a general operator ${\bf  A}$, we refer to its domain by $\cD({\bf  A})$.

Often, we work with powers ${\bf L}^n$ of operators, e.g.~we consider the left-definite operator induced by the expression $\ell^n$. Abusing notation, we use ${\bf L}\ti{max}^n = \{\ell^n, \cD\ti{max}^n\}$ where $\cD\ti{max}^n:=\cD\ti{max}(\ell^n)$. Further, we let $[\fdot,\fdot]$ denote a general sesquilinear form, and $[\fdot,\fdot]_n$ stand for the sesquilinear form associated with $\ell^n$.

We generally have $(m,m)$ be the deficiency indices of ${\bf L}\ti{min}$. We note that the deficiency indices of ${\bf L}\ti{min}^n$ then amount to $(nm,nm)$.\\

\noindent{\bf Acknowledgement.} Many thanks to F.~Gesztesy for making useful suggestions.

\section{Background}\label{s-background}
Consider the classical Sturm--Liouville differential equation
\begin{align*}
\dfrac{d}{dx}\left[p(x)\dfrac{dy}{dx}\right]+q(x)y=-\lambda w(x)y,
\end{align*}
where $y$ is a function of the independent variable $x$, $p(x),w(x)>0$ a.e.~on $(a,b)$ and $q(x)$ real-valued a.e.~on $(a,b)$. 
Furthermore, $1/p(x),q(x),w(x)\in L^1\ti{loc}[(a,b),dx]$. Additional details about Sturm--Liouville theory can be found in \cite{GZ}.
This differential expression can be viewed as a linear operator, mapping a function $f$ to the function $\ell[f]$ via
\begin{align}\label{d-sturmop}
\ell[f](x):=-\dfrac{1}{w(x)}\left(\dfrac{d}{dx}\left[p(x)\dfrac{df}{dx}(x)\right]+q(x)f(x)\right).
\end{align}
This expression can be viewed as an unbounded operator acting on Hilbert space $L^2[(a,b),w]$, endowed with the inner product 
$
\langle f,g\rangle:=\int_a^b f(x)\overline{g(x)}w(x)dx.
$
In this setting, the eigenvalue problem $\ell[f](x)=\lambda f(x)$ can be considered. The operators of interest, $\{\ell,L^2[(a,b),w]\}$, are assumed to possess a set of orthogonal eigenfunctions that is complete in the domain. The expression $\ell[\fdot]$ defined in equation \eqref{d-sturmop} has been well-studied, see \cite{I} for an in-depth discussion of its relation to orthogonal polynomials. However, the operator $\{\ell,L^2[(a,b),w]\}$ is not self-adjoint a priori. Subsection \ref{s-gkn} details the imposition of boundary conditions to ensure self-adjointness.

Furthermore, the operator $\ell^n[\fdot]$ is defined as the operator $\ell[\fdot]$ composed with itself $n$ times, creating a differential operator of order $2n$. Every formally symmetric differential expression $\ell^n[\fdot]$ of order $2n$ with coefficients $a_k:(a,b)\to\RR$ and $a_k\in C^k(a,b)$ for $k=0,1,\dots,n$ and $n\in\NN$ has the {\em Lagrangian symmetric form} 

\begin{align}\label{e-lagrangian}
\ell^n[f](x)=\sum_{j=1}^n(-1)^j(a_j(x)f^{(j)}(x))^{(j)}, \text{ } x\in(a,b).
\end{align}
Further details can be found in \cite{DS, LWOG}.

The classical differential expressions of Jacobi, Hermite, and Laguerre all admit such a representation, and are semi-bounded. Semi-boundedness is defined as the existence of a constant $k\in\RR$ such that for all $x$ in the domain of the operator ${\bf A}$ the following inequality holds:
$$\langle {\bf A}x,x\rangle\geq k\langle x,x\rangle.$$
This additional property, combined with self-adjointness, allows for a continuum of nested Hilbert spaces to be defined within $L^2[(a,b),w]$ via the expressions $\ell^n[\fdot]$. Indeed, this continuum is a Hilbert scale, and many facts about the spectrum and the operators can be deduced using this point of view (e.g.~\cite{DHS, LW13}). More details about Hilbert scales can be found in \cite{AK, KP}. This particular Hilbert scale with self-adjoint operators that are semi-bounded is the topic in left-definite theory \cite{LW02}, part of which we explain in Subsection \ref{s-leftdef}.\\
\indent The combination of left-definite theory applied to differential expressions and self-adjoint extension theory allows for the necessary discussion of boundary conditions for these operators in later sections. Another necessary technicality is the use of the graph norm, in Subsection \ref{s-graph}. This subsection allows the maximal domain that the examined differential expression is defined on, to be decomposed into an orthogonal direct sum. This decomposition is essential to the results in Section \ref{s-results}.

\subsection{Left-Definite Theory}\label{s-leftdef}
Left-definite theory deals primarily with the spectral theory of Sturm--Liouville differential operators. The terminology itself can be traced back to Weyl in 1910 \cite{W}. A general framework for the left-definite theory of bounded-below, self-adjoint operators in a Hilbert space wasn't developed until 2002 in the landmark paper by Littlejohn and Wellman \cite{LW02}. Specifically, the left-definite theory allows one to generate a scale of operators (by composition), each of which possess the same spectrum as the original. 

Let $\cV$ be a vector space over $\CC$ with inner product $\langle\fdot,\fdot\rangle$ and norm $||\fdot||$. The resulting inner product space is denoted $(\cV,\langle\fdot,\fdot\rangle)$.

\begin{defn}[{\cite[Theorem 3.1]{LW02}}] \label{t-ldinpro}
Suppose ${\bf A}$ is a self-adjoint operator in the Hilbert space $\cH=(\cV,\langle\fdot,\fdot\rangle )$ that is bounded below by $kI$, where $k>0$. Let $r>0$. Define $\cH_r=(\cV_r, \langle\fdot,\fdot\rangle_r)$ with
$$\cV_r=\cD ({\bf A}^{r/2})$$
and
$$\langle x,y\rangle_r=\langle {\bf A}^{r/2}x,{\bf A}^{r/2}y\rangle \text{ for } (x,y\in \cV_r).$$
Then $\cH_r$ is said to be the $r$th {\em left-definite space} associated with the pair $(\cH,{\bf A})$.
\end{defn}

It was proved in \cite[Theorem 3.1]{LW02} that $\cH_r=(\cV_r, \langle\fdot,\fdot\rangle)$ is also described as the left-definite space associated with the pair $(\cH, {\bf A}^r)$, and we call $\cH_r$ the {\em r}th {\em left-definite space associated with the pair} $(\cH,{\bf A})$.
Specifically, we have:
\begin{enumerate}
\item $\cH_r$ is a Hilbert space,
\item $\cD ({\bf A}^r)$ is a subspace of $\cV_r$,
\item $\cD ({\bf A}^r)$ is dense in $\cH_r$,
\item $\langle x,x\rangle_r\geq k^r\langle x,x\rangle$ ($x\in \cV_r$), and
\item $\langle x,y\rangle_r=\langle {\bf A}^rx,y\rangle$ ($x\in\cD ({\bf A}^r)$, $y\in \cV_r$).
\end{enumerate}

The left-definite domains are defined as the domains of compositions of the self-adjoint operator ${\bf A}$, but the operator acting on this domain is slightly more difficult to define. 

\begin{defn}
Let $\cH=(\cV,\langle\fdot,\fdot\rangle)$ be a Hilbert space. Suppose ${\bf A}:\cD ({\bf A})\subset \cH\to \cH$ is a self-adjoint operator that is bounded below by $k>0$. Let $r>0$. If there exists a self-adjoint operator ${\bf A}_r:\cH_r\to \cH_r$ that is a restriction of ${\bf A}$ from the domain $\cD({\bf A})$ to $\cD({\bf A}^r)$,
we call such an operator an $r$th {\em left-definite operator associated with $(\cH,{\bf A})$}.
\end{defn}

The connection between the $r$th left-definite operator and the $r$th composition of the self-adjoint operator ${\bf A}$ is now made explicit. 

\begin{cor}[{\cite[Corollary 3.3]{LW02}}] \label{t-comppower}
Suppose ${\bf A}$ is a self-adjoint operator in the Hilbert space $\cH$ that is bounded below by $k>0$. For each $r>0$, let $\cH_r=(\cV_r, \langle\fdot,\fdot\rangle_r)$ and ${\bf A}_r$ denote, respectively, the $r$th left-definite space and the $r$th left definite operator associated with $(\cH,{\bf A})$. Then
\begin{enumerate}
\item $\cD ({\bf A}^r)=\cV_{2r}$, in particular, $\cD ({\bf A}^{1/2})=\cV_1$ and $\cD ({\bf A})=\cV_2$;
\item $\cD ({\bf A}_r)=\cD ({\bf A}^{(r+2)/2})$, in particular, $\cD ({\bf A}_1)=\cD ({\bf A}^{3/2})$ and $\cD ({\bf A}_2)=\cD ({\bf A}^2)$.
\end{enumerate}
\end{cor}

The left-definite theory is particularly important for self-adjoint differential operators that are bounded below, as they are generally unbounded. The theory is trivial for bounded operators, as shown in {\cite[Theorem 3.4]{LW02}}.

Our applications of left-definite theory will be focused on differential operators which possess a complete orthogonal set of eigenfunctions in $\cH$. In {\cite[Theorem 3.6]{LW02}} it was proved that the point spectrum of ${\bf A}$ coincides with that of ${\bf A}_r$, and similarly for the continuous spectrum and for the resolvent set. 
It is possible to say more, a complete set of orthogonal eigenfunctions will persist throughout each space in the Hilbert scale.

\begin{theo}[{\cite[Theorem 3.7]{LW02}}] \label{t-leftdefortho}
If $\{\f_n\}_{n=0}^{\infty}$ is a complete orthogonal set of eigenfunctions of ${\bf A}$ in $\cH$, then for each $r>0$, $\{\f_n\}_{n=0}^{\infty}$ is a complete set of orthogonal eigenfunctions of the $r$th left-definite operator ${\bf A}_r$ in the $r$th left-definite space $\cH_r$.
\end{theo}

Another perspective on the last theorem is that it gives us a valuable indicator for when a space is a left-definite space for a specific operator.

On the side we note that left-definite theory can be extended to bounded below operators by applying shifts. Uniqueness is then given up to the chosen shift.

A description of these left-definite spaces in terms of standard boundary conditions on a Hilbert space has been noticeably missing, despite the broad framework and range of results described above. This isn't to say that there are no descriptions of the left-definite spaces, just that they are not classically expressed by GKN theory.

\begin{ex*}
Let ${\bf A}$ denote the usual self-adjoint operator with the Laguerre polynomials as a complete set of orthogonal eigenfunctions $\{\f_n\}_{n=0}^{\infty}$. For $\al>-1$ and $j\in\NN_0$, let $L_{\al+j}^2(0,\infty)$ be the Lebesgue space with norm induced by the inner product $\int_0^{\infty}f(t)\overline{g(t)}t^{\al+j}e^{-t}dt$. 
The $n$th left-definite Hilbert space associated with the pair $(\cH,{\bf A}) = (L^2_{\al}(0,\infty),{\bf A})$, also possessing this complete set of eigenfunctions, is defined as $\cH_n=(\cV_n,\langle\fdot,\fdot\rangle_n)$, where
$$
\cV_n:=\left\{f:(0,\infty)\to\CC ~~\bigg|~~ f\in\text{AC}\ti{loc}^{(n-1)}(0,\infty);~ f^{(n)}\in L_{\al+n}^2(0,\infty)\right\}
$$
and
$$
\langle p,q\rangle_n:=\sum_{j=0}^n b_j(n,k)\int_0^{\infty} p^{(j)}(t)\overline{q^{(j)}(t)}t^{\al+j}e^{-t}dt ~~\text{ for } (p,q\in\mathcal{P}),
$$
where $\mathcal{P}$ is the space of all (possibly complex-valued) polynomials. The constants $b_j(n,k)$ are defined as
$$
b_j(n,k):=\sum_{i=0}^j \dfrac{(-1)^{i+j}}{j!}\binom{j}{i}(k+i)^n.
$$

This description of a specific left-definite space is only included as a reference for the complexity of the results in this paper, and details can be found in \cite{LW02}.
\hfill$\kreuz$
\end{ex*}

\subsection{General Glazman--Krein--Naimark (GKN) Theory}\label{s-gkn}
There is a vast amount of literature concerning the extensions of symmetric operators. Here we present only that which pertains to self-adjoint extensions and applications to GKN theory. This will be primarily applied to linear differential operators.

\begin{defn}[variation of {\cite[Section 14.2]{N}}]\label{d-defect}
For a a symmetric, closed operator ${\bf  A}$ on a Hilbert space $\cH$, define 
the {\bf positive defect space} and the {\bf negative defect space}, respectively, by
$$\cD_+:=\left\{f\in\cD({\bf  A}^*)~~|~~{\bf  A}^*f=if\right\}
\qquad\text{and}\qquad
\cD_-:=\left\{f\in\cD({\bf  A}^*)~~|~~{\bf  A}^*f=-if\right\}.$$
\end{defn}

On the side we note that, in light of {\cite[Theorem XII.4.8]{DS}}, we can assume without loss of generality that all considered operators are closed because we are concerned exclusively with self-adjoint extensions of symmetric operators.

We are most interested in the dimensions dim$(\cD_+)=m_+$ and dim$(\cD_-)=m_-$, which are called the {\bf positive} and {\bf negative deficiency indices of ${\bf  A}$}, respectively. These dimensions are usually conveyed as the pair $(m_+,m_-)$. 
The deficiency indices of $T$ correspond to how far from self-adjoint ${\bf  A}$ is. A symmetric operator ${\bf  A}$ has self-adjoint extensions if and only if its deficiency indices are equal {\cite[Section 14.8.8]{N}}.

\begin{theo}[{\cite[Theorem 14.4.4]{N}}]\label{t-decomp}
If ${\bf  A}$ is a closed, symmetric operator, then the subspaces $\cD_{\bf  A}$, $\mathcal{D}_+$, and $\mathcal{D}_{-}$ are linearly independent and their direct sum coincides with $\cD_{{\bf  A}^*}$, i.e.,
$$\cD_{{\bf  A}^*}=\cD_{\bf  A}\dotplus\mathcal{D}_+ \dotplus\mathcal{D}_{-}.$$
(Here, subspaces $\cX_1, \cX_2, \hdots ,\cX_p$ are said to be {\bf linearly independent}, if $\sum_{i=1}^p x_i = 0$ for $x_i\in \cX_i$ implies that all $x_i=0$.)
\end{theo}

\subsection{GKN Theory and Sturm--Liouville operators}
Let $\ell[\fdot]$ be a Sturm--Liouville differential expression on some Hilbert space $L^2[(a,b),w]$ as in \eqref{d-sturmop}. Furthermore, let $\ell[\fdot]$ generate an expression $\ell^n[\fdot]$ of order $2n$ via composition, $n\in\NN$. The analysis of self-adjoint extensions does not involve changing the differential expression associated with the operator at all, merely the domain of definition, by applying boundary conditions. 

\begin{defn}[{\cite[Section 17.2]{N}}]\label{d-max}
The {\bf maximal domain} of $\ell^n[\fdot]$ is given by 
\begin{align*}
\cD^n\ti{max}=\cD\ti{max}(\ell^n):=\bigg\{f:(a,b)\to\mathbb{C}~~\bigg|~~f^{(k)}(x)\in\text{AC}\ti{loc}(a,b),~~k=&0,1,\dots,2n-1;
\\
&
f,\ell^n[f]\in L^2[(a,b),w]\bigg\}.
\end{align*}
\end{defn}

The designation of ``maximal'' is appropriate in this case because $\cD\ti{max}(\ell^n)$ is the largest possible subspace for which $\ell^n$ maps back into $L^2[(a,b),w]$. For $f,g\in\cD\ti{max}(\ell^n)$ and $a<\al\le \beta<b$ the {\bf sesquilinear form} associated with $\ell^n$ by
\begin{equation}\label{e-greens}
[f,g]_n\bigg|_{\al}^{\beta}:=\int_{\al}^{\beta}\left\{\ell^n[f(x)]\overline{g(x)}-\ell^n[\overline{g(x)}]f(x)\right\}w(x)dx.
\end{equation}
The equation \eqref{e-greens} is {\bf Green's formula} for $\ell^n[\fdot]$, and is an equivalent definition to the classical one from Sturm--Liouville theory utilizing Wronskians \cite[Equation (3.5)]{LWOG}.

\begin{theo}[{\cite[Section 17.2]{N}}]\label{t-limits}
The limits $[f,g]_n(b):=\lim_{x\to b^-}[f,g]_n(x)$ and $[f,g]_n(a):=\lim_{x\to a^+}[f,g]_n(x)$ exist and are finite for $f,g\in\cD\ti{max}(\ell^n)$.
\end{theo}

\begin{defn}[{\cite[Section 17.2]{N}}]\label{d-min}
The {\bf minimal domain} of $\ell^n[\fdot]$ is given by
\begin{align*}
\cD^n\ti{min}=\cD\ti{min}(\ell^n)=\{f\in\cD\ti{max}(\ell^n)~~|~~[f,g]_n\big|_a^b=0~~\forall g\in\cD\ti{max}(\ell^n)\}.
\end{align*}
\end{defn}

The maximal and minimal operators associated with the expression $\ell^n[\fdot]$ are defined as ${\bf L}^n\ti{min}=\{\ell^n,\cD^n\ti{min}\}$ and ${\bf L}^n\ti{max}=\{\ell^n,\cD^n\ti{max}\}$ respectively. By {\cite[Section 17.2]{N}}, these operators are adjoints of one another, i.e.~$({\bf L}^n\ti{min})^*={\bf L}^n\ti{max}$ and $({\bf L}^n\ti{max})^*={\bf L}^n\ti{min}$.

In the context of differential operators, we work with the a special case of Theorem \ref{t-decomp}:

\begin{theo}[{\cite[Section 14.5]{N}}]
Let $\cD^n\ti{max}$ and $\cD^n\ti{min}$ be the maximal and minimal domains associated with the differential expression $\ell^n[\fdot]$, respectively. Then, for $n\in\NN$,
\begin{equation}\label{e-vN}
\cD^n\ti{max}=\cD^n\ti{min}\dotplus\cD^n_+\dotplus\cD^n_-.
\end{equation}
\end{theo}

Equation \eqref{e-vN} is commonly known as {\bf von Neumann's formula}. Here $\dotplus$ denotes the direct sum, and $\cD^n_+,\cD^n_-$ are the defect spaces associated with the expression $\ell^n[\fdot]$. The decomposition can be made into an orthogonal direct sum by using the graph norm, see Section \ref{s-graph}.

From {\cite[Section 14.8.8]{N}} we know that, if the operator ${\bf L}^n\ti{min}$ has any self-adjoint extensions, then the deficiency indices of ${\bf L}^n\ti{min}$ have the form $(m,m)$, where $0\leq m\leq 2n$ and $2n$ is the order of $\ell^n[\fdot]$.
Akhiezer and Glazman \cite{AG} have shown that the number $m$ can take on any value between $0$ and $2n$. In regards to differential expressions, the order of the operator is greater than or equal to each of the two deficiency indices by necessity.  Hence, Sturm--Liouville expressions that generate self-adjoint operators have deficiency indices $(0,0)$, $(1,1)$ or $(2,2)$. This is related to the discussion of an expression being limit-point or limit-circle at endpoints, see \cite{BDG,DS,H,N} for more details. 

In order to formulate the GKN theorems, we recall an extension of linear independence to one that mods out by a subspace. This subspace will be the minimal domain in applications.

\begin{defn}[{\cite[Section 14.6]{N}}]\label{d-linind}
Let $\cX_1$ and $\cX_2$ be subspaces of a vector space $\cX$ such that $\cX_1\le \cX_2$. Let $\{x_1,x_2,\dots,x_r\}\subseteq \cX_2$. We say that $\{x_1,x_2,\dots,x_r\}$ is {\bf linearly independent modulo $\cX_1$} if
$$\sum_{i=1}^r\al_ix_i\in \cX_1 \text{ implies } \al_i=0\text{ for all }i=1,2,\dots, r.$$
\end{defn}

The following two theorems form the core of GKN theory.

\begin{theo}[GKN1,~{\cite[Theorem 18.1.4]{N}}]\label{t-gkn1}
Let ${\bf L}^n=\{\ell^n,\cD_{{\bf L}}^n\}$ be a self-adjoint extension of the minimal operator ${\bf L}^n\ti{min}=\{\ell^n,\cD^n\ti{min}\}$ with deficiency indices $(m,m)$. Then the domain $\cD_{{\bf L}}^n$ consists of the set of all functions $f\in\cD^n\ti{max}$, which satisfy the conditions
\begin{equation}\label{e-gkn1a}
[f,w_k]_n\bigg|_a^b=0, \text{ }k=1,2,\dots,m ,
\end{equation}
where $w_1,\dots,w_m\in \cD^n\ti{max}$ are linearly independent modulo $\cD^n\ti{min}$ for which the relations
\begin{equation}\label{e-gkn1b}
[w_j,w_k]_n\bigg|_a^b=0, \text{ }j,k=1,2,\dots,m
\end{equation}
hold.
\end{theo}

The requirements in equation \eqref{e-gkn1b} are commonly referred to as {\bf Glazman symmetry conditions}. The converse of the GKN1 Theorem is also true.

\begin{theo}[GKN2,~{\cite[Theorem 18.1.4]{N}}]\label{t-gkn2}
Assume we are given arbitrary functions $w_1,w_2,\dots,w_m\in\cD^n\ti{max}$ which are linearly independent modulo $\cD^n\ti{min}$ and which satisfy the relations \eqref{e-gkn1b}. Then the set of all functions $f\in\cD^n\ti{max}$ which satisfy the conditions \eqref{e-gkn1a} is domain of a self-adjoint extension of ${\bf L}^n\ti{min}$.
\end{theo}

These two theorems completely answer the question of how boundary conditions can be used to create self-adjoint extensions. Applications of this theory hinge on determining the proper $w_k$'s that will define the domain of the desired self-adjoint extension.

\subsection{Test for linear independence modulo the minimal domain}\label{s-TestLI}
Consider a symmetric expression $\ell$ with deficiency indices $(m,m)$ on the Hilbert space $L^2[(a,b),w]$.
The following simple result will be used to test for linear independence modulo $\cD\ti{min}$. The main idea is to find a sufficient condition in terms of a certain matrix of sesquilinear forms (corresponding to $\ell$) having full rank.

\begin{prop}\label{p-LI}
Given vectors $w_1, \hdots, w_r\in \cD\ti{max}$, $r\le 2m$. Assume that the $r\times r$ matrix ${\bf M}$ with entries ${\bf M}_{ik} = [w_i,w_k] |_a^b$ for $1\le i,k \le r$
has full rank. Then $w_1, \hdots, w_r$ are linearly independent modulo $\cD\ti{min}$.
\end{prop}

We will prove this proposition in a moment.

\begin{rem*}
Proposition \ref{p-matrix} shows that the converse of Proposition \ref{p-LI} does not hold.
\end{rem*}

In virtue of Linear Algebra (dimension counting, and realizing that removing vectors from a basis leaves behind a linearly independent set) we obtain an immediate consequence, which both Sections \ref{s-framework} and \ref{s-legendre} rely heavily on.

\begin{cor}\label{corollary}
Given vectors $w_1, \hdots, w_r\in \cD\ti{max}$, $r\le 2m$. Assume that for some vectors $w_{r+1}, \hdots, w_{s}\in \cD\ti{max}$, $r\le s \le 2m$, the $s\times s$ matrix ${\bf M}$ with entries ${\bf M}_{ik} = [w_i,w_k] |_a^b$ (for $1\le i,k \le s$)
has full rank. Then $w_1, \hdots, w_{s}$ are linearly independent modulo $\cD\ti{min}$, and so are the vectors $w_1, \hdots, w_r$.
\end{cor}

\begin{rem*}
In our applications below, we usually have $r=m$ and $s=2m$. Moreover, $w_1, \hdots, w_m$ will be eigenfunctions, and $w_{m+1}, \hdots, w_{2m}$ will be functions of the second kind.
\end{rem*}

\begin{proof}[Proof of Proposition \ref{p-LI}]
Our goal is to show that the set $w_1,\dots,w_r$ is linearly independent modulo the minimal domain $\cD\ti{min}$. To that end, suppose
\begin{align}\label{e-inDmin}
\sum_{k=1}^{r}\al_kw_k\in\cD\ti{min}.
\end{align}
We want to show that $\al_k=0$ for all $k=1,\dots,r$.

The definition of the minimal domain says that $y\in\cD\ti{min}$ if and only if $[y,w]|_{a}^{b}=0$ for all $w\in\cD\ti{max}$. Letting $w = \sum_{k=1}^{r}\al_kw_k$ and using the linearity of the sesquilinear form, we see that \eqref{e-inDmin} implies
\begin{align}\label{equationX}
\sum_{k=1}^{r}\al_k[w_i,w_k]\bigg|_{a}^b=0 \text{ for }i=1,\dots,r.
\end{align}

Now, interpreting \eqref{equationX} for a specific $i$ as the $i$th row of a matrix equation, we see that \eqref{equationX} is equivalent to the matrix equation
\[
{\bf M}{\bf \alpha} = {\bf 0}
\quad\text{with}\quad
{\bf M} = \left(
\begin{matrix}[w_1,w_1]\big|_a^b & \dots & [w_1,w_{r}]\big|_a^b \\ \vdots & \ddots & \vdots \\ [w_{r},w_1] \big|_a^b& \dots & [w_{r},w_{r}]\big|_a^b\end{matrix}\right),
\quad
{\bf \alpha} = \left(
\begin{matrix}\alpha_1\\
\vdots\\
\alpha_r
\end{matrix}\right),
\]
and the zero vector ${\bf 0}\in \R^r$.

And since we assume that ${\bf M}$ has full rank, we conclude that $\alpha_k=0$ for all $k=1,\hdots, r$.
\end{proof}

\subsection{Graph Norm}\label{s-graph}

Let ${\bf  A}$ be a densely defined symmetric operator on a separable Hilbert space $\cH$. Furthermore, for $x,y\in\cD({\bf  A}^*)$, denote the {\bf graph inner product} by $$\langle x,y\rangle\ci{{\bf  A}}:=\langle  x,y\rangle\ci{\cH}+\langle {\bf  A}^*x, {\bf  A}^*y\rangle\ci{\cH}.$$
This section will use the convention that $\cD({\bf  A}^*)$ has the topology defined by the {\bf graph norm} $||x||\ci{{\bf  A}}:=\langle x,x\rangle\ci{{\bf  A}}^{1/2}$ induced by the inner product $\langle x,y\rangle\ci{{\bf  A}}$, unless the contrary is explicitly stated. We note that the closure, $\overline{{\bf  A}}$, of ${\bf  A}$ is the restriction of ${\bf  A}^*$ to the closure of $\cD({\bf  A})$ in the Hilbert space $\cD({\bf  A}^*)$ under the graph norm. The graph norm allows for some more elegant results in the theory of self-adjoint extensions, and will be central to our main theorems.

\begin{lem}[{\cite[Lemma XII.4.10]{DS}}]\label{l-graphnorm} Using the above conventions, we have that:
\begin{enumerate}
\item $\cD(\overline{{\bf  A}})$, $\cD_+$, and $\cD_-$ are closed orthogonal subspaces of the Hilbert space $\cD({\bf  A}^*)$.
\item $\cD({\bf  A}^*)=\cD(\overline{{\bf  A}})\oplus\ci{{\bf  A}}\cD_+\oplus\ci{{\bf  A}}\cD_-$. Here $\oplus\ci{{\bf  A}}$ denotes the orthogonal sum, with respect to the graph inner product.
\end{enumerate}
\end{lem}

The results in this paper heavily rely on this orthogonal decomposition so a proof of the Lemma, following that of \cite{DS}, is included.

\begin{proof}
The space $\cD(\overline{{\bf  A}})$ is closed by the above note, while $\cD_+$ and $\cD_-$ are closed because they are finite dimensional. Since $\cD_+$ and $\cD_-$ are clearly linear subspaces of $\cD({\bf  A}^*)$, it remains to show that the spaces $\cD(\overline{{\bf  A}})$, $\cD_+$, and $\cD_-$ are mutually orthogonal, and that their sum is $\cD({\bf  A}^*)$.

Suppose $d\in\cD(\overline{{\bf  A}})$, $d_+\in\cD_+$, and $d_-\in\cD_-$. We will show that $\langle d,d_+\rangle\ci{{\bf  A}}=\langle d,d_-\rangle\ci{{\bf  A}}=\langle d_-,d_+\rangle\ci{{\bf  A}}=0$. First, since ${\bf  A}^*\supseteq\overline{{\bf  A}}$ we compute
\begin{align*}
\langle d,d_+\rangle\ci{{\bf  A}}&=\langle d,d_+\rangle+\langle {\bf  A}^*d,{\bf  A}^*d_+\rangle=\langle d,d_+\rangle+\langle \overline{{\bf  A}}d,{\bf  A}^*d_+\rangle \\
&=\langle d,d_+\rangle+\langle \overline{{\bf  A}}d,id_+\rangle=\langle d,d_+\rangle+\langle d,i\overline{{\bf  A}}^*d_+\rangle \\
&=\langle d,d_+\rangle+\langle d,i{\bf  A}^*d_+\rangle=\langle d,d_+\rangle+\langle d,i^2d_+\rangle=0.
\end{align*}
Similarly, $\langle d,d_-\rangle\ci{{\bf  A}}=0$. Next,
\begin{align*}
\langle d_-,d_+\rangle\ci{{\bf  A}}&=\langle d_-,d_+\rangle+\langle {\bf  A}^*d_-,{\bf  A}^*d_+\rangle \\
&=\langle d_-,d_+\rangle+\langle -id_-,id_+\rangle=0.
\end{align*}
Hence the spaces $\cD(\overline{{\bf  A}})$, $\cD_+$, and $\cD_-$ are mutually orthogonal, and $\cD(\overline{{\bf  A}})\oplus\ci{{\bf  A}}\cD_+\oplus\ci{{\bf  A}}\cD_-$ is contained in $\cD({\bf  A}^*)$.

To show that they are equal, we will show that zero is the only vector orthogonal to the three subspaces involved. Suppose $v$ is orthogonal to $\cD(\overline{{\bf  A}})$, $\cD_+$, and $\cD_-$. Then $0=\langle d,v\rangle\ci{{\bf  A}}=\langle d,v\rangle+\langle {\bf  A}^*d,{\bf  A}^*v\rangle$, for all $d$ in $\cD(\overline{{\bf  A}})$. Hence $\langle d,v\rangle=-\langle {\bf  A}^*d,{\bf  A}^*v\rangle=-\langle {\bf  A}d,{\bf  A}^*v\rangle$ because ${\bf  A}$ is symmetric and so ${\bf  A}={\bf  A}^*$ on $\cD(\overline{{\bf  A}})$. 

Recall that $\langle \fdot,v\rangle$ is a continuous linear functional on the dense subset $\cD(\overline{{\bf  A}})$ of the original Hilbert space $\cH$. By the definition of the adjoint, we see that ${\bf  A}^*v$ is in $\cD({\bf  A}^*)$ and so ${\bf  A}^*({\bf  A}^*v)=-v$. Hence, we have $(I+{\bf  A}^*{\bf  A}^*)v=(I+i{\bf  A}^*)(I-i{\bf  A}^*)v=0$. And so we obtain ${\bf  A}^*[(I-i{\bf  A}^*)v]=i(I-i{\bf  A}^*)v$, or $(I-i{\bf  A}^*)v\in\cD_+$. Also, if $d_+\in\cD_+$, then
\begin{align*}
0&=\langle v,d_+\rangle\ci{{\bf  A}}=\langle v,d_+\rangle+\langle {\bf  A}^*v,{\bf  A}^*d_+\rangle=\langle v,d_+\rangle+\langle {\bf  A}^*v,id_+\rangle \\
&=\langle v,d_+\rangle-i\langle {\bf  A}^*v,d_+\rangle=\langle (I-i{\bf  A}^*)v,d_+\rangle.
\end{align*}
Since $(I-i{\bf  A}^*)v$ is in $\cD_+$, this implies that $(I-i{\bf  A}^*)v=0$. Hence ${\bf  A}^*v=-iv$, or $v\in\cD_-$. But $\langle \cD_-,v\rangle\ci{{\bf  A}}=0$. Hence $v=0$. Therefore $\cD({\bf  A}^*)=\cD(\overline{{\bf  A}})\oplus\ci{{\bf  A}}\cD_+\oplus\ci{{\bf  A}}\cD_-$.
\end{proof}

The previous lemma can be viewed with regards to differential operators by using that the maximal and minimal operators are adjoints of one another. Our differential operators are assumed to be closed, so it is possible to replace $\cD(\overline{{\bf  A}})$ with $\cD\ti{min}$ and $\cD({\bf  A}^*)$ with $\cD\ti{max}$, while still assuming that $\cD\ti{max}$ is endowed with the graph norm. Hence, the lemma translates into von Neumann's formula \eqref{e-vN}, but the decomposition is orthogonal due to the different norm. The $r/2$ graph norm of ${\bf A}$ will be of particular interest, and is denoted by $\langle x,y\rangle\ci{{\bf A}^{r/2}}=\langle x,y\rangle\ci{\cH}+\langle {\bf A}^{r/2}x,{\bf A}^{r/2}y\rangle\ci{\cH}$. The operator ${\bf A}$ is assumed to be self-adjoint and bounded below by $k>0$, as in Section \ref{s-leftdef}.
\begin{lem} \label{t-norms}
The $rth$ left-definite norm is equivalent to the $r/2$ graph norm. Concretely,
\[
\|x\|_r^2\le
\|x\|\ci{{\bf A}^{r/2}}^2\le
C\|x\|_r^2,
\]
where the constant $C$ depends on $k$ and $r$. 
\end{lem}
\begin{proof}
The $r$th left-definite inner product can be defined via Theorem \ref{t-ldinpro} as $||x||^2_r=\langle x,x\rangle_r=\langle {\bf A}^{r/2}x,{\bf A}^{r/2}x\rangle\ci{\cH}$. Also, recall that the definition of the $r$th left-definite space implied the stipulation $\langle x,x\rangle_r\geq k^r\langle x,x\rangle\ci{\cH}$. Then,
\begin{align*}
||x||^2\ci{{\bf A}^{r/2}}=\langle {\bf A}^{r/2}x,{\bf A}^{r/2}x\rangle\ci{\cH} +\langle x,x\rangle\ci{\cH}=||x||^2_r+||x||^2\ci{\cH}\geq ||x||^2_r,
\end{align*}
and
\begin{align*}
||x||^2_r&
=\frac{1}{2}\langle x,x\rangle_r+\frac{1}{2}\langle x,x\rangle_r
\\&
\geq \frac{k^r}{2}\langle x,x\rangle\ci{\cH} +\frac{1}{2}\langle x,x\rangle_r \\
&=\frac{1}{2}\left[k^r\langle x,x\rangle\ci{\cH} +\langle {\bf A}^{r/2}x,{\bf A}^{r/2}x\rangle\ci{\cH}\right] \\
&
\ge \frac{1}{2}\min\{k^r, 1\}||x||^2\ci{{\bf A}^{r/2}}.
\end{align*}
Furthermore, no problems arise by passing to the $r/2$ graph norm instead of the usual $r$ graph norm. This is because the domain of the $r/2$ graph norm coincides with the $r/2$th left-definite space and $\cD ({\bf A}_{r/2})\supset\cD ({\bf A}_r)$. We conclude that the two norms are indeed equivalent.
\end{proof}

\section{Glazman--Krein--Naimark (GKN) Conditions using Eigenfunctions}\label{s-framework}
Consider a symmetric expression $\ell$ with deficiency indices $(m,m)$ on the Hilbert space $L^2[(a,b),w]$. Let ${\bf L} = \{\ell,\cD_{\bf L}\}$ be a self-adjoint extension. 
Assume that the domain of ${\bf L}$ includes a complete set of orthogonal eigenfunctions, say $\{P_k\}_{k=0}^{\infty}$.
The GKN1 Theorem (Theorem \ref{t-gkn1}) states that all self-adjoint extensions are obtained by imposing $m$ GKN conditions on the maximal domain. GKN conditions are induced by functions $w_1,\dots,w_{m}$, which satisfy three conditions:
\begin{enumerate}
\item[(C1)] The functions $w_1,\dots,w_{m}$ must be linearly independent modulo the minimal domain.
\item[(C2)] The complete system of orthogonal eigenfunctions must be included in the domain, pursuant to equation \eqref{e-gkn1a}.
\item[(C3)] The functions $w_1,\dots,w_{m}$ must satisfy the Glazman symmetry conditions in equation \eqref{e-gkn1b}.
\end{enumerate}
Later we will choose $w_1,\dots,w_{m}$ to be eigenfunctions. In that case, the first item (C1) implies both (C2) and (C3).

\begin{rem}\label{r-GKNsymmetry}
In fact, item (C3) can in general be obtained from (C1), if we allow for an insignificant modification of the $w_1,\dots,w_{m}$. Indeed, the Glazman symmetry conditions (C3) are easily attained by taking linear combinations of vectors satisfying (C1) via a procedure that is similar to the Gram--Schmidt orthogonalization, using the sesquilinear form instead of an inner product. We notice that such linear combinations will not change (C1). They will span the same domain modulo $\cD\ti{min}$. So they will also not change property (C2).
\end{rem}

By assumption we have
$
\dim (\cD_+\dotplus\cD_-)
=
2m.$
A basis of this space $\text{mod}\,(\cD\ti{min})$ would be ideal. But, if $w_1,\dots,w_{m}$ are linearly independent modulo the minimal domain, then they can be completed to a basis $w_1,\dots,w_{2m}$; and vice versa.

Consider the matrix
\begin{align}\label{e-M}
\sbox0{$\begin{matrix}[w_1,w_1] & \dots & [w_1,w_{m}] \\ \vdots & \ddots & \vdots \\ [w_{m},w_1] & \dots & [w_{m},w_{m}]\end{matrix}$}
\sbox1{$\begin{matrix}[w_1,w_{{m}+1}] & \dots & [w_1,w_{2{m}}] \\ \vdots & \ddots & \vdots \\ [w_{m},w_{{m}+1}] & \dots & [w_{m},w_{2{m}}]\end{matrix}$}
\sbox2{$\begin{matrix}[w_{{m}+1},w_1] & \dots & [w_{{m}+1},w_{m}] \\ \vdots & \ddots & \vdots \\ [w_{2{m}},w_1] & \dots & [w_{2{m}},w_{m}]\end{matrix}$}
\sbox3{$\begin{matrix}[w_{{m}+1},w_{{m}+1}] & \dots & [w_{{m}+1},w_{2{m}}] \\ \vdots & \ddots & \vdots \\ [w_{2{m}},w_{{m}+1}] & \dots & [w_{2{m}},w_{2{m}}]\end{matrix}$}
{\bf M}=\left(
\begin{array}{c|c}
\usebox{0}&\usebox{1}\vspace{-.35cm}\\\\
\hline\vspace{-.35cm}\\
 \vphantom{\usebox{0}}\usebox{2}&\usebox{3}
\end{array}
\right),
\end{align}
where each sesquilinear form is evaluated from $a$ to $b$.

We formulate sufficient conditions under which eigenfunctions act as GKN conditions. Though the set up seems less natural, the result will be shown to bear useful consequences. The theorem shows that, under the assumption that eigenfunctions $\{P_{k_i}\}_{i=1}^{m}$ can be completed to a basis of $\cD_+\dotplus\cD_-$, these eigenfunctions are then appropriate GKN conditions.

\begin{theo}\label{t-eigenfunctions}
Let ${\bf L} = \{\ell,\cD_{\bf L}\}$ be a self-adjoint operator on the Hilbert space $L^2[(a,b),w]$, and be an extension of a minimal symmetric operator that has deficiency indices $(m,m)$. Assume that $\cD_{\bf L}$ includes a complete set of orthogonal eigenfunctions, $\{P_k\}_{k=0}^{\infty}$. Furthermore, assume that a basis modulo $\cD\ti{min}$ of the defect spaces $\cD_+\dotplus\cD_-$ is given by the collection $\{P_{k_1},\dots,P_{k_m},f_1,\dots,f_m\}$.

Then $\cD_{\bf L}$ is given by imposing $\{P_{k_1},\dots,P_{k_m}\}$ as GKN conditions on $\cD\ti{max}$.
\end{theo}

Before we prove this result, we take an excursion via two corollaries.

In combination with Corollary \ref{corollary}, a slight modification of the proof of Theorem \ref{t-eigenfunctions} immediately yields a similar result for general $w_1,\dots,w_{2m}$ possibly apart from the less significant (as was explained in Remark \ref{r-GKNsymmetry}) GKN conditions (C2).

\begin{cor}\label{c-last}
Consider a symmetric operator with expression $\ell$ on $L^2[(a,b),w]$ that has deficiency indices $(m,m)$.
If ${\bf M}$ defined as in equation \eqref{e-M} has full rank for some choice of $w_1,\dots,w_{2m}\in \cD\ti{max}(\ell)$, then any subset of $m$ of these induces GKN conditions so long as we drop the symmetry condition \emph{(C2)}.
\end{cor}

In Corollary \ref{c-last}, we do not claim that these conditions induce a particular self-adjoint extension, but rather just one of the infinitely many possible ones, again, not expecting the GKN symmetry condition (C2).

In the next section, we will apply another immediate consequence of the Theorem to powers of a self-adjoint extension associated with the Legendre expression. Recall that we work with a symmetric expression $\ell$ with deficiency indices $(m,m)$ on the Hilbert space $L^2[(a,b),w]$, and with a self-adjoint extension ${\bf L} = \{\ell,\cD_{\bf L}\}$. 
Recall that we assume that the domain of ${\bf L}$ includes a complete set of orthogonal eigenfunctions, say $\{P_k\}_{k=0}^{\infty}$.
Assuming that ${\bf L}$ is bounded below, we can consider the self-adjoint operator (associated with the differential expression $\ell^n$) that arises from the $2n$th left-definite domain for ${\bf L}$. 
Consider the matrix ${\bf M}_{nm}$ where the sesquilinear forms $[\fdot,\fdot]$ are the ones corresponding to $\ell^n$, which we generally denote by $[\fdot,\fdot]_n$. 

\begin{cor}\label{c-fullrank}
Assume that $nm$ of the $w_1, \hdots, w_{2nm}$ functions are eigenfunctions (e.g.~$w_k=P_k$ for $k=1,\hdots, nm$). If the corresponding $(2nm)\times (2nm)$ matrix ${\bf M}$ from equation \eqref{e-M} has full rank, then the domain of the $2n$th left-definite operator can be represented by imposing GKN conditions with those eigenfunctions on the maximal domain. 
\end{cor}

Examples in Section \ref{s-legendre} show that, for small values of $n$, the first ${m}$ eigenfunctions can in fact serve as half of these basis vectors. Therefore, they can be used to form GKN conditions.

In accordance with equation \eqref{e-orthoeigens} below, orthogonal functions which also satisfy the eigenvalue equation, are automatically symmetric in the sense of Glazman's condition (C2).
So, the entire problem of imposing appropriate GKN conditions on symmetric operators to yield left-definite self-adjoint extensions is reduced to showing that the upper-right quadrant of the matrix ${\bf M}$ has full rank. We explore this further below in Proposition \ref{p-matrix}.

\begin{proof}[Proof of Theorem \ref{t-eigenfunctions}]
The collection $\{P_{k_1},\dots,P_{k_m},f_1,\dots,f_m\}$ forms a basis of $\cD_+\dotplus\cD_-$, so all self-adjoint extensions of the minimal operator come from using $m$ distinct GKN conditions written as:
\begin{align*}
G_i=a_{i,1}P_{k_1}+a_{i,2}P_{k_2}+\dots+a_{i,m}P_{k_m}+a_{i,m+1}f_1+a_{i,m+2}f_2+\dots+a_{i,2m}f_m,
\end{align*}
for $i=1,\dots,m$. The GKN1 Theorem \ref{t-gkn1} implies that 
\begin{align*}
\cD_{\bf L}=\{f\in\cD\ti{max} ~|~ [f,G_i]|_a^b=0, ~i=1,\dots,m \text{ for some choice of } a_{i,j}\text{'s}\}.
\end{align*}

The claim is then that $a_{i,j}=0$ for all $j>m$. In particular, this choice of constants $a_{i,j}$'s needs to include the subset of orthogonal eigenfunctions $\{P_{k_1},\dots,P_{k_m}\}$ in the domain. Notice that an application of Green's formula for the sesquilinear form yields 
\begin{align}\label{e-orthoeigens}
[P_i,P_j]\bigg|_a^b&=\int_a^b\ell[P_i]P_j wdx-\int_a^bP_i \ell[P_j]wdx
=(\lambda_i-\lambda_j)\int_a^bP_i P_j wdx=0.
\end{align}
Here, if $i=j$ then $\lambda_i-\lambda_j=0$, and if $i\neq j$ then we use the orthogonality of $P_i$ and $P_j$.

The Glazman symmetry conditions (C3) follow immediately.

Fix the index $i$ and test the chosen $G_i$ against these orthogonal eigenfunctions in the sesquilinear form as follows:
\begin{align*}
0=[P_{k_1},G_1]|_a^b=0+\dots+0+&a_{i,m+1}[P_{k_1},f_1]|_a^b+\dots+a_{i,2m}[P_{k_1},f_m]|_a^b, \\
0=[P_{k_2},G_1]|_a^b=0+\dots+0+&a_{i,m+1}[P_{k_2},f_1]|_a^b+\dots+a_{i,2m}[P_{k_2},f_m]|_a^b, \\
&\vdots \\
0=[P_{k_m},G_1]|_a^b=0+\dots+0+&a_{i,m+1}[P_{k_m},f_1]|_a^b+\dots+a_{i,2m}[P_{k_m},f_m]|_a^b.
\end{align*}
This problem can be recast in terms of the upper half of the finite $2m\times 2m$ matrix ${\bf M}$ as above. The upper-left quadrant of ${\bf M}$ is then $0$. Explicitly, this problem represents the upper-right quadrant of ${\bf M}$ multiplied by a column vector of $a_{i,j}$'s. However, the collection $\{P_{k_1},\dots,P_{k_m},f_1,\dots,f_m\}$ constitutes a basis for $\cD_+\dotplus\cD_-$. Hence, the upper-right quadrant of ${\bf M}$ has full rank. The only way to yield the necessary column vectors of $0$'s, in the above equations, is for all of the $a_{i,j}$'s to be $0$. Hence, $a_{i,j}=0$ for all $j>m$. The calculation was for a general fixed $i$ so necessarily
\begin{align*}
G_i=a_{i,1}P_{k_1}+a_{i,2}P_{k_2}+\dots+a_{i,m}P_{k_m} \text{ for all }i=1,\dots,m.
\end{align*}
The $G_i$'s themselves must also be linearly independent modulo the minimal domain. So, modulo $\mathcal{D}\ti{min}$, their span is identical to that of $\{P_{k_1}, \hdots , P_{k_m}\}$.
\end{proof}

\section{GKN Conditions for Powers of Legendre}\label{s-legendre}
When considering a specific differential operator, the problem of finding GKN conditions can be written rather explicitly in terms of the matrix ${\bf M}$ from the previous section.

The explicit tools developed here can be adapted to study spectral theory for powers of other Bochner--Krall polynomial systems.
Here we focus our attention on the powers of the Legendre operator.

On the one hand, these examples expand the observations in \cite{LW15}. On the other hand, they motivate and tie into our main results, see Section \ref{s-results}. The method established here is more explicit than the abstract approach in Section \ref{s-results}.

The investigation into the boundary conditions associated with left-definite theory begins by considering the classical Legendre differential operator ${\bf L}=\{\ell, \cD_{\bf L}\}$ on the Hilbert space $L^2(-1,1)$, given by
\begin{align}\label{e-legendre}
\ell[y](x)=-((1-x^2)y'(x))'
\end{align}
together with the domain
\begin{align}\label{e-defect11}
\cD_{\bf L} &= \{f\in \cD\ti{max};
(1-x^2)f'(x)\mid_{-1}^1
= 0
\},
\end{align}
which contains the Legendre polynomials $\{P_k\}_{k=0}^\infty$, see e.g.~\cite{LWOG}. Recall that $ \cD\ti{max}$ was provided in Definition \ref{d-max}. Both $\lim_{x\to\pm1}(1-x^2)f'(x)$ exist by Theorem \ref{t-limits}, and because $(1-x^2)f'(x)\mid_{-1}^1 = [f,1]\mid_{-1}^1$.

\begin{rem*}
In \cite{LWOG}, it was proved that $\cD_{\bf L}$ is equal to the domain induced by left-definite theory. The equality of such domains for ${\bf L}^n$ is discussed in Section \ref{s-conj}.
\end{rem*}

This operator possesses the Legendre polynomials $P_k(x)$, $k\in\NN_0$, as a complete set of eigenfunctions. That is, the polynomial $y(x)=P_k(x)$ is a solution of the eigenvalue equation 
\begin{align}\label{e-eigen}
\ell[y](x)=k(k+1)y(x),
\end{align}
for each $k$ we have $P_k\in \cD_{\bf L}$ and $\spa\{P_k\}_{k=0}^\infty$ is dense in $L^2(-1,1)$.

Left-definite theory allows for the construction of a sequence of Hilbert spaces whose domains are operated on by integer composition powers of ${\bf L}$. It is no hindrance that for odd powers of composition, we encounter fractional left-definite spaces, e.g.~when $n=3$ the operator ${\bf L}^3$ corresponds to $\cV_{3/2}$. The case $n=2$ has been investigated by Littlejohn and Wicks in \cite{LW15} and \cite{LWOG}. They showed that the left-definite domain in this case could be defined via the GKN1 Theorem using $w_1\equiv 1$ and $w_2(x)=x$ in equation \eqref{e-gkn1a}. The literature does not provide results for $n\ge 3$.

Since the minimal Legendre operator ${\bf L}\ti{min} = (\ell, \cD\ti{min})$ has deficiency indices $(1,1)$ (as is visible from \eqref{e-defect11}), the powers ${\bf L}^n\ti{min} = (\ell^n, \cD\ti{min}^n)$ of the Legendre operator have deficiency indices $(n,n)$. In other words, we have $m=n$ here.

In the remainder of this section we present a few results for general $n$, while focusing on some special cases for more explicit results. We are mostly interested in $n=3$, but also work with $n=4$ and $n=5$, and include some new observations when $n=2$.

\subsection{The structure of the sesquilinear matrix ${\bf M}_n$}
These explicit results can be extended by analogy to larger values of $n$. For general $n$, the matrix ${\bf M}_n$ will be $(2n)\times (2n)$ and the entries are sesquilinear forms corresponding to expression $\ell^n$ in Green's formula. If $n$ is even, choose $P_0,P_1,\hdots,P_{n-1}$ and $Q_0,Q_1,\hdots,Q_{n-1}$ (Legendre functions of the second kind). If $n$ is odd, choose $P_0,P_1,\hdots,P_{n-1}$ and $Q_1,Q_2,\hdots,Q_n$. These will be the choice for basis candidates unless otherwise stated. In the following example we explain why these are good choices.

\begin{ex*}
Let $n=3$. The first two Legendre polynomials are $P_0\equiv 1$ and $P_1(x)=x$, so a reasonable guess in this case would be to try to use as GKN conditions
\[
P_0\equiv 1,\quad
P_1(x)=x,\quad
P_2(x)=\frac{1}{2}(3x^2-1).
\]
It will be shown that these GKN conditions indeed yield the desired domain. The most difficult condition to prove is the linear independence modulo the minimal domain. 
For $n=3$, the deficiency indices are $(3,3)$. So we have $m=n=3$ and a basis of $\cD_+^3\dotplus\cD^3_-$ will be $2n=6$ dimensional.

To show that a set $w_1,w_2,\dots,w_6$ is linearly independent modulo $\cD\ti{min}^3$, follow the setup of the matrix described in Section \ref{s-framework} to yield
\[
\sbox0{$\begin{matrix}[w_1,w_1] & [w_1,w_2] & [w_1,w_3] \\ [w_2,w_1] & [w_2,w_2] & [w_2,w_3] \\ [w_3,w_1] & [w_3,w_2] & [w_3,w_3]\end{matrix}$}
\sbox1{$\begin{matrix}[w_1,w_4] & [w_1,w_5] & [w_1,w_6] \\ [w_2,w_4] & [w_2,w_5] & [w_2,w_6] \\ [w_3,w_4] & [w_3,w_5] & [w_3,w_6]\end{matrix}$}
\sbox2{$\begin{matrix}[w_4,w_1] & [w_4,w_2] & [w_4,w_3] \\ [w_5,w_1] & [w_5,w_2] & [w_5,w_3] \\ [w_6,w_1] & [w_6,w_2] & [w_6,w_3]\end{matrix}$}
\sbox3{$\begin{matrix}[w_4,w_4] & [w_4,w_5] & [w_4,w_6] \\ [w_5,w_4] & [w_5,w_5] & [w_5,w_6] \\ [w_6,w_4] & [w_6,w_5] & [w_6,w_6]\end{matrix}$}
{\bf M}_3=\left(
\begin{array}{c|c}
\usebox{0}&\usebox{1}\vspace{-.35cm}\\\\
\hline\vspace{-.35cm}\\
 \vphantom{\usebox{0}}\usebox{2}&\usebox{3}
\end{array}
\right),
\]
where $[\fdot,\fdot]$ stands for the sesquilinear form in equation \eqref{e-sesqui} below for $n=3$ and is evaluated from $-1$ to $1$ by taking limits.

For one moment, let us assume that all those limits exist. Now the idea is that we will be choosing all of the $w_j$ to be either Legendre eigenfunctions or Legendre functions of the second kind, all of which satisfy eigenvalue equations $\ell[y]=\lambda y$. The representation through Green's formula
\[
[w_j,w_k]\bigg|_{-1}^1=\int_{-1}^1\ell^3[w_j]w_kdx-\int_{-1}^1w_j\ell^3[w_k]dx
\]
will be of use.

The Legendre functions of the second kind are commonly denoted by $Q_k$, $k\in\NN_0$. The explicit representations for the first four of them are:
\begin{align*}
Q_0(x)&=\dfrac{1}{2}\ln\left(\dfrac{1+x}{1-x}\right), \quad
\\
Q_1(x)&=\dfrac{x}{2}\ln\left(\dfrac{1+x}{1-x}\right)-1, \quad
\\
Q_2(x)&=\dfrac{3x^2-1}{4}\ln\left(\dfrac{1+x}{1-x}\right)-\dfrac{3x}{2},\text{ and}
\\
Q_3(x)&=\dfrac{5x^3-3x}{4}\ln\left(\dfrac{1+x}{1-x}\right)-\dfrac{5x^2}{2}+\dfrac{2}{3}.
\end{align*}
More information can be found in \cite{D,P}.

Explicitly, take
\[
w_1=P_0, \quad w_2=P_1, \quad w_3=P_2,\quad \text{and}\quad
w_4=Q_1, \quad w_5=Q_2, \quad w_6=Q_3.
\]
Let us ensure that the limits in these sesquilinear forms are well-defined.
\hfill$\kreuz$
\end{ex*}

We turn back to general $n$. Since $P_k$ are eigenfunctions, they trivially belong to $\cD^n\ti{max}$.
In conjunction with Theorem \ref{t-limits} the following result shows that the limits of the sesquilinear forms $\lim_{x\to -1^+}[w_j,w_k](x)$  and $\lim_{x\to 1^-}[w_j,w_k](x)$ both exist when $w_j$ and $w_k$ are eigenfunctions or polynomials of the second kind.

\begin{prop}\label{p-SecondMax}
The Legendre functions of the second kind are in the maximal domain $\cD\ti{max}^n$ for every value of $n$.
\end{prop}

\begin{proof}
By Definition \ref{d-max} the maximal domain of $\ell^n$ is given by 
\begin{align*}
\cD^n\ti{max}=\big\{y:(-1,1)\to\mathbb{C}~~|~~y^{(k)}\in\text{AC}\ti{loc}(-1,1),~~k=&0,1,\dots,2n-1;\, 
y,\ell^n[y]\in L^2(-1,1)\big\}.
\end{align*}
These functions are in $L^2(-1,1)$ and their integrals are explicitly known \cite{D}:
\begin{align*}
\int_{-1}^1 (Q_k(x))^2dx=\dfrac{\pi^2-2(1+\cos^2(k\pi))\psi'(k+1)}{2(2k+1)},
\end{align*}
where the function $\psi(x)=\Gamma'(x)/\Gamma(x)$ is the so-called digamma function. Derivatives of the functions $Q_k$ have singularities only at $-1$ and $1$, so $Q_k^{(r)}\in C^{\infty}[\alpha,\beta]$ for all $r\in\NN$ and all $[\alpha,\beta]\subset(-1,1)$. Hence, each derivative is itself locally continuously differentiable and locally absolutely continuous. The fact that the functions $Q_k$ are solutions to the eigenvalue equation \eqref{e-eigen} trivially implies that $\ell^n[Q_k]\in L^2(-1,1)$ for all $k\in\NN_0$.
\end{proof}

Initial conclusions can be made about the structure of ${\bf M}_n$ using technical facts about the entries.
In the remainder of this section we will always assume that the first $n$ of the $w_j$ are Legendre polynomials and the others are Legendre functions of the second kind (with indices in $\NN_0$).

\begin{prop}\label{p-matrix}
The matrix ${\bf M}_n$ is antisymmetric and takes the form
\[
{\bf M}_n=\left(
\begin{array}{c|c}
$\large \bf{0}$ & $\large ${\bf B}_n$$
\vspace{-.4cm}\\\\
\hline\vspace{-.3cm}\\
  $\large $-{\bf B}_n^\top $$ & $\large ${\bf C}_n$$
\end{array}
\right)
\]
so that $\det{({\bf M}_n)}=\det{({\bf B}_n^\top {\bf B}_n)}=[\det{({\bf B}_n)}]^2$.
In particular, if ${\bf B}_n$ has full rank, then the $w_j$ used to produce the entries of ${\bf M}_n$ form a basis of $\cD\ti{max}$ modulo $\cD\ti{min}$.
\end{prop}

\begin{proof}
Simplify the entries of the matrix ${\bf M}$ by using Green's formula \eqref{e-greens} as follows: For $f_j = P_j$ or $f_j = Q_j$ and $g_k = P_k$ or $g_k =Q_k$ and with the standard inner product $\langle\fdot,\fdot\rangle$ on $L^2(-1,1)$, notice that
\begin{align}\label{e-jandk} 
[f_j,g_k]\bigg|_{-1}^1&=
\langle \ell^n[f_j] , g_k\rangle - \langle f_j , \ell^n[g_k]\rangle
=[j^n(j+1)^n-k^n(k+1)^n]\langle f_j , g_k\rangle
\\
&\notag
 =-[k^n(k+1)^n-j^n(j+1)^n]\langle f_j , g_k\rangle
 = \langle f_j , \ell^n[g_k]\rangle -\langle \ell^n[f_j] , g_k\rangle
 \\
 &\notag
 =-[g_k,f_j]\bigg|_{-1}^1
.
\end{align}
Here, we used that repeated applications of the eigenvalue equation \eqref{e-eigen} yield $\ell^n[P_k]=k^n(k+1)^nP_k$, and similarly for $Q_k$.
This shows antisymmetry.

In the first quadrant of the matrix, we encounter the case when the sesquilinear form is evaluated for two Legendre polynomials. Due to the orthogonality of the inner product in \eqref{e-jandk} evaluates to zero when $f_j=P_j$ and $g_k=P_k$ when $j\neq k$. For $j=k$, the coefficient in front of the inner product equals zero. Summing up, this shows
 \[
 [P_j,P_k]\bigg|_{-1}^1
 =0, \text{ for all }j,k\in\NN_0.
 \]
 
By linear algebra, the determinant of ${\bf M}$ can be calculated using
\[
\det{({\bf M}_n)}=\det{[({\bf 0}^\top)({\bf C}_n)-(-{\bf B}^\top _n)({\bf B}_n)]}=\det{({\bf B}^\top _n{\bf B}_n)}=[\det({\bf B}_n)]^2.
\]

The last statement follows from Corollary \ref{c-fullrank}.
\end{proof}

Several immediate consequences of equation \eqref{e-jandk} regarding the particular entries of ${\bf M}_n$ are now apparent. As before, let $f_j = P_j$ or $f_j = Q_j$ and $g_k = P_k$ or $g_k =Q_k$.
First notice that
\[
[f_j,g_k]\bigg|_{-1}^1 = 0
\qquad\text{for } j=k.
\]

Fortunately, formulas are known to calculate the relevant inner products \cite[pp.~236]{P}, and also exist in the cases where two different $P_j$'s and two different $Q_k$'s are considered: 
When both functions are Legendre functions of the second kind it is known that
\begin{align*}
\langle Q_j,Q_k\rangle=\dfrac{[\psi(j+1)-\psi(k+1)][1+\cos(j\pi)\cos(k\pi)]+\frac{1}{2}\pi\sin((k-j)\pi)}{(k-j)(j+k+1)},
\end{align*}
for real $j,k>0$ with $j\neq k$. The case where the functions are of mixed type is given by
\begin{align*}
\langle P_j,Q_k\rangle=\dfrac{2\sin(j\pi)\cos(k\pi)[\psi(j+1)-\psi(k+1)]+\pi\cos((k-j)\pi)-\pi}{\pi(k-j)(j+k+1)},
\end{align*}
and is valid for real $j,k>0$ and $j\neq k$.
The function $\psi(x)$ is the digamma function in the above two formulas.

The formulas are only necessary for $j,k\in\NN_0$, reducing the computations significantly.
In particular, for $j,k\in\NN_0$:
\begin{align}\label{e-qs}
&[Q_j,Q_k]\bigg|_{-1}^1\\
\notag
&=
\left\{
\begin{array}{ll}
\dfrac{2[\psi(j+1)-\psi(k+1)][j^3(j+1)^3-k^3(k+1)^3]}{(k-j)(j+k+1)}=:\Phi_{jk},&j+k \text{ even and }j\neq k,\\
0,& j+k \text{ odd or }j=k.
\end{array}
\right.
\end{align}
Analogously, for $j,k\in\NN_0$:
\begin{align}
\label{e-pq}
[P_j,Q_k]\bigg|_{-1}^1&=
\left\{
\begin{array}{ll}
\dfrac{-2[j^3(j+1)^3-k^3(k+1)^3]}{(k-j)(j+k+1)}=:\f_{jk},&j+k \text{ odd},\\
0,& j+k \text{ even}.
\end{array}
\right.
\end{align}

The definition of the matrix entries given in equations \eqref{e-qs} and \eqref{e-pq} immediately yield that $\Phi_{jk}=-\Phi_{kj}$ and $\f_{jk}=-\f_{kj}$. In particular, the blocks ${\bf B}_n$ and ${\bf C}_n$ from Proposition \ref{p-matrix} can be explicitly computed for even $n$ as 
\[
{\bf B}_n=
\begin{pmatrix}
0 & \f_{01} & 0 & \f_{03} & \dots & \f_{0(n-1)} \\
\f_{10} & 0 & \f_{12} & 0 & \dots & 0 \\
0 & \f_{21} & 0 & \f_{23} & \dots & \f_{2(n-1)} \\
\f_{30} & 0 & \f_{32} & 0 & \dots & 0 \\
0 & \f_{41} & 0 & \f_{43} & \dots & \f_{4(n-1)} \\
\vdots & \vdots & \vdots & \vdots & \ddots & \vdots \\
\f_{(n-1)0} & 0 & \f_{(n-1)2} & 0 & \dots & 0
\end{pmatrix},
\]
and
\[
{\bf C}_n=
\begin{pmatrix}
0 & 0 & \Phi_{02} & 0 & \dots & 0 \\
0 & 0 & 0 & \Phi_{13} & \dots & \Phi_{1(n-1)} \\
\Phi_{20} & 0 & 0 & 0 & \dots & 0 \\
0 & \Phi_{31} & 0 & 0 & \dots & \Phi_{3(n-1)} \\
0 & 0 & \Phi_{42} & 0 & \dots & 0 \\
\vdots & \vdots & \vdots & \vdots & \ddots & \vdots \\
0 & \Phi_{(n-1)1} & 0 & \Phi_{(n-1)3} & \dots & 0
\end{pmatrix}.
\]

The case where $n$ is odd can similarly be written down in terms of the above formulas.

Mathematica can be used to ease the trouble of populating the matrix with the relevant entries.

\begin{ex*}
We return to the case $n=3$. The following matrix is the result of setting $w_1=P_0$, $w_2=P_1$, $w_3=P_2$, $w_4=Q_1$, $w_5=Q_2$, and $w_6=Q_3$:
\[
\sbox0{$\begin{matrix}8 & 0 & 288\\0 & 104 & 0\\104 & 0 & 504\end{matrix}$}
\sbox1{$\begin{matrix}-8 & 0 & -104\\0 & -104 & 0\\-288 & 0 & -504\end{matrix}$}
\sbox2{$\begin{matrix}0 & 0 & \frac{860}{3}\\0 & 0 & 0\\\frac{-860}{3} & 0 & 0\end{matrix}$}
{\bf M}_3=\left(
\begin{array}{c|c}
\makebox[\wd0]{\large \bf{0}}&\usebox{0}\vspace{-.4cm}\\\\
\hline\vspace{-.3cm}\\
  \vphantom{\usebox{0}}\usebox{1}&\usebox{2}
\end{array}
\right).
\]
The reduced row echelon form of the above matrix is easily calculated to be $I_6$. Indeed, in accordance with Proposition \ref{p-matrix}, it suffices to observe that the upper right quadrant ${\bf B}_3$ of ${\bf M}_3$ has full rank, which it clearly does.

Now, Corollary \ref{c-fullrank} asserts that these are in fact the desired GKN conditions for the cube of the Legendre differential operator.
\hfill$\kreuz$
\end{ex*}

In virtue of Proposition \ref{p-matrix} higher values of $n$ are more easily accessible, as they are less computationally expensive and the digamma function in the lower right quadrant is avoided. 

The cases where $n=4$ and $n=5$ are included below to illustrate how this method can be generalized.

\begin{ex*}
For $n=4$ choose the functions $P_0,\dots,P_3$ and $Q_0,\dots,Q_3$ as candidates for basis vectors. The relevant matrix can be computed to be:
\[
{\bf B}_4=
\begin{pmatrix}
0 & 16 & 0 & 3456 \\
16 & 0 & 640 & 0 \\
0 & 640 & 0 & 6480 \\
3456 & 0 & 6480 & 0
\end{pmatrix}.
\]
This matrix is of particular interest, because it is representative of all cases where $n$ is even and the basis vectors are chosen to be $P_0,\dots,P_{n-1}$ and $Q_0,\dots,Q_{n-1}$. In these cases, the matrix ${\bf B}_n$ is symmetric.

The invertibility of ${\bf B}_4$ immediately reduces to showing that both submatrices
\[
\begin{pmatrix}
16 & 3456 \\
 640 &  6480
\end{pmatrix}
\quad\text{and}\quad
 \begin{pmatrix}
16  & 640  \\
3456  & 6480 
\end{pmatrix}
\]
have non-zero determinant. This is trivially true.

Therefore, the GKN conditions $P_0,\dots,P_3$ yield a self-adjoint operator and since all eigenfunctions $P_k$ satisfy these conditions, we obtain the left-definite operator that is associated with ${\bf L}^4$.
\hfill$\kreuz$\end{ex*}

\begin{ex*}
For $n=5$ choose the functions $P_0,\dots,P_4$ and $Q_1,\dots,Q_5$ as candidates for basis vectors. The relevant matrix can be computed to be:
\[
{\bf B}_5=
\begin{pmatrix}
32 & 0 & 41472 & 0 & 1620000 \\
0 & 3872 & 0 & 355552 & 0 \\
3872 & 0 & 80352 & 0 & 2024352 \\
0 & 80352 & 0 & 737792 & 0 \\
355552 & 0 & 737792 & 0 & 4220000 \\
\end{pmatrix}.
\]
Again, it can easily be shown that ${\bf B}_5$ has full rank.
\hfill$\kreuz$\end{ex*}

\begin{ex*}
Matlab has allowed us to verify that the first $n$ Legendre polynomials are suitable for ${\bf L}^n$ when $n\le16$. Vast stratification in magnitude of matrix entries is responsible for this \emph{early} failing of the numerical computations.
\hfill$\kreuz$\end{ex*}

\subsection{Necessary Condition and Conjecture}
It is clear the methods developed above are powerful, albeit limited to calculation, and can be expressed in more generality. Apart from increasing $n$, another way to generalize stems from choosing a more general set of indices for the Legendre polynomials, and for the Legendre functions of the second kind, as hinted at in Theorem \ref{t-eigenfunctions}.

There is a necessary condition for the truncated matrix ${\bf B}_n$ having full rank which requires some additional requirements on the indices of both the Legendre polynomials and the Legendre functions of the second kind.

\begin{prop}\label{t-evenodd}
Let ${\bf M}_n$ define a basis of functions for the space $\cD_+^n\dotplus\cD_-^n$ when $n$ is general. Let the choice of functions for this basis be $P_{j_1},P_{j_2},\dots,P_{j_n}$ and $Q_{k_1},Q_{k_2},\dots,Q_{k_n}$. Define $\mathfrak{I}:=\{j_1,\dots,j_n,k_1,\dots,k_n\}$ be the collection of indices for these functions. Then $\mathfrak{I}$ contains $n$ even and $n$ odd elements.
\end{prop}

\begin{proof}
Assume ${\bf M}_n$ has full rank, so that $P_{j_1},P_{j_2},\dots,P_{j_n}$ and $Q_{k_1},Q_{k_2},\dots,Q_{k_n}$ are indeed a basis for $\cD\ti{max}$ modulo $\cD\ti{min}$. Without loss of generality, also assume that there are more even numbers in $\mathfrak{I}$ than odd numbers.
By Proposition \ref{p-matrix}, the Glazman conditions and anti-symmetry of the matrix then completely reduces our problem to showing that ${\bf B}_n$ does not have full rank.

Recall that the $(i,l)$-entry of ${\bf B}_n$ equals $[P_{j_i},Q_{k_l}]|_{-1}^1$ for $1\le i,l\le n$ and is given by equation \eqref{e-pq}. So, these entries are only nonzero when $j_i+k_l$ is odd. Hence, interchange rows of ${\bf B}_n$ to group the even indices first for the set of $P_{j_i}$'s, and interchange columns to group the even indices first for the set of $Q_{k_l}$'s. This creates two blocks of entries on the anti-diagonal in the upper right quadrant, where the $P$ index is even and the $Q$ index is odd, and one where the $P$ index is odd and the $Q$ index is even. The rank of ${\bf B}_n$ is the sum of the rank of these two blocks. However, because there are more even indices than odd ones, neither of these blocks are square, so the sum of their ranks cannot be equal to $n$. The means ${\bf M}_n$ cannot have rank $2n$, which is a contradiction. Therefore, the number of even and odd numbers in the set $\mathfrak{I}$ must be equal.
\end{proof}

The matrix ${\bf B}_n$ raises several important questions concerning the rules that are necessary and/or sufficient on the indices (other than the requirement of Proposition \ref{t-evenodd} above) in order to ensure that the ${\bf B}_n$ has rank $n$. Unfortunately, the answer to this question can only be conjectured currently. Progress in this direction using theoretical aspects of the setup is shown in the next two sections.

\begin{conj}
Let $P_{j_1},P_{j_2},\dots,P_{j_n}$ be any set of $n$ distinct Legendre polynomials, with $n_1$ odd indices and $n_2$ even indices so that $n_1+n_2=n$. Then these Legendre polynomials can be used as GKN conditions to define the $n/2$ left-definite domain.

In particular, choose any $n$ distinct Legendre functions of the second kind $Q_{k_1},Q_{k_2},\dots,Q_{k_n}$ with $n_2$ odd indices and $n_1$ even indices. Then together these $2n$ functions constitute a basis of the space $\cD_+^n\dotplus\cD_-^n$.
\end{conj}

Immediate inspiration for the conjecture stems from the $n=2$ case.

\begin{ex*}
Consider the $n=2$ case for simplicity. If one odd index and one even index are chosen for the $P_j$'s (say $P_j$ and $P_k$ with even $j$ and odd $k$) then the claim follows: ${\bf B}_2$ only has entries on the anti-diagonal and hence is rank 2.
\hfill$\kreuz$
\end{ex*}

\begin{ex*}
Now, again for $n=2$, assume that both chosen indices are odd, so that both of the indices for the $Q_k$'s are even. As a further simplification, choose $Q_0$ and $Q_2$. The matrix of interest is 
\[
{\bf B}_2=
\begin{pmatrix}
\f_{j0} & \f_{j2} \\
\f_{k0} & \f_{k2}
\end{pmatrix},
\]
where $j$ and $k$ are both odd.

To reduce to row echelon form, the row operation necessary is $\al\f_{j0}+\f_{k0}=0$ so that $\al=-\f_{k0}/\f_{j0}$. This changes $\f_{k2}$ to be $\widetilde{\f}_{k2}=\f_{j2}+\al\f_{k2}$. This can be written out explicitly to be 
\begin{align*}
\widetilde{\f}_{k2}=&((k^2 (k + 1)^2)-1) (j^2 (j + 1)^2 - 36) (j^2 + j) (-k^2 - k + 
     6) \\ 
     &- ((j^2 (j + 1)^2)-1) (k^2 (k + 1)^2 - 36) (k^2 + k) (-j^2 - j + 
     6).
\end{align*}
Mathematica shows that $\widetilde{\f}_{k2}=0$ has no solutions for distinct $j$ and $k$, where both are odd and positive. A similar equation is relevant for the case where both $j$ and $k$ are even and $Q_1$ and $Q_3$ are chosen as the paired basis vectors. Mathematica similarly shows it is not possible to get $j$ and $k$ to be distinct, even and positive. This shows that any 2 distinct indices for the Legendre polynomials will be sufficient to define the first left-definite domain via GKN conditions.  
\hfill$\kreuz$\end{ex*}

There is further evidence that the conjecture is true.

\begin{ex*}
Let $n=4$ and choose the functions $P_{17},P_{42},P_{49},P_{125}$ and $Q_{24},Q_{82},Q_{97},Q_{178}$ as candidates for the basis vectors. The relevant matrix can be computed to be:
\[
{\bf B}_4=
\begin{pmatrix}
821988432 & 660210828928 & 0 & 65319097828480 \\
0 & 0 & 2118187203328 & 0 \\
38811250000 & 968624405632 & 0 & 70078111267456 \\
8123415750000 & 13280257143232 & 0 & 120291674577856
\end{pmatrix}.
\]
It is not hard to see that this matrix possesses full rank. 
\hfill$\kreuz$
\end{ex*}

Unfortunately, the complexity of the matrix operations and higher values for $n$ mean verifying that solutions are not of the desired form is computationally expensive. However, to add a little more weight to the conjecture, the above form can be easily adapted to the $n=3,4$ cases where there are two even or two odd choices of indices of the $P_j$'s. In conclusion the assertion in the conjecture about the choice of indices for the $P_j$'s is true for $n=2$ and for special cases when $n=3$ and when $n=4$.

\section{General Left-Definite Theory yields GKN Conditions}\label{s-results}

\indent The following theorems apply to general left-definite settings so it is imperative to clarify some of the subtler points of abstraction. It should be understood that the differential expression that is being generated from left-definite theory is not changing under the classical extension theory, only that the minimal domain is being augmented to include more functions and become self-adjoint. Specifically, ${\bf L}\ti{min}$ and ${\bf L}$ possess domains $\cD\ti{min}$ and $\cD_{{\bf L}}$ respectively, but both operate on functions using $\ell[\fdot]$. Also, recall that an operator defined by left-definite theory means that it is generated by composing self-adjoint differential operators with themselves to create a Hilbert scale of operators. The domains of these operators shrink as the number of compositions increase. For more details refer to Section \ref{s-leftdef}.

\begin{theo}
Let ${\bf L}$ be a self-adjoint operator defined by left-definite theory on $L^2[(a,b),w]$ with domain $\cD_{{\bf L}}$ (which is a restriction of the maximal domain $\cD\ti{max}$) that includes a complete orthogonal system of eigenfunctions. Enumerate the eigenfunctions as $\{P_k\}_{k=0}^{\infty}$. Furthermore, let ${\bf L}$ be an extension of the minimal (symmetric, closed) operator ${\bf L}\ti{min}$ with domain $\cD\ti{min}$, where ${\bf L}$ and ${\bf L}\ti{min}$ operate on their respective domains by $\ell[\fdot]$, and ${\bf L}\ti{min}$ has deficiency indices $(m,m)$.  

Then, the GKN conditions for the self-adjoint operator ${\bf L}$ are given by some $\{P_{k_1},\dots,P_{k_m}\}$.
\end{theo}

\begin{proof}
The case $m=0$ is trivial, since then ${\bf L}\ti{min} = {\bf L} = {\bf L}\ti{max}$. Let $m\ge1$.

The main work lies in showing that some of the eigenfunctions yield appropriate choices for $w_1,\dots,w_m$ from the perspective of Theorem \ref{t-gkn1} (the GKN1 Theorem). That is, we need to show that there are $m$ members $\{P_{k_1},\dots,P_{k_m}\}$ of the collection $\{P_k\}_{k=0}^{\infty}$, which both satisfy the Glazman symmetry conditions \eqref{e-gkn1b} and are linearly independent modulo the minimal domain $\cD\ti{min}$, see Definition \ref{d-linind}.

First, we show that all possible choices of $\{P_{k_1},\dots,P_{k_m}\}$ satisfy the Glazman symmetry conditions. Recall that an application of the Green's formula for the sesquilinear form yields 
$[P_i,P_j]\big|_a^b=0$ for $i,j\in \N_0$, see equation \eqref{e-orthoeigens}.

Next, we show that there exist functions $\{P_{k_1},\dots,P_{k_m}\}$ that are linearly independent modulo $\cD\ti{min}$.
Part (2) of Lemma \ref{l-graphnorm} states that the maximal domain decomposes orthogonally with respect to graph norm into $\cD\ti{max}=\cD\ti{min}\oplus\ci{{\bf  A}}\cD_+\oplus\ci{{\bf  A}}\cD_-$.
Define the auxiliary functions $\{\widetilde{P}_k\}_{k=0}^{\infty}$ to be the orthogonal projection (in accordance with the graph norm) of the $P_k$'s onto $\cD_+\oplus\ci{{\bf  A}} \cD_-$.

Recall that ${\bf L}\ti{min}$ has defect indices $(m,m)$ and that the subspaces satisfy $$\cD\ti{min}<\cD\ti{\bf L}<\cD\ti{max}.$$ Since ${\bf L}$ is self-adjoint, we have that $\cD_{\bf L}\ominus\ci{{\bf  A}}\cD\ti{min}$ is $m$ dimensional. The projection is orthogonal, so 
$
\{\widetilde{P}_k\}_{k=0}^{\infty}
$
spans an $m$ dimensional subspace of $\cD_+\oplus\ci{{\bf  A}} \cD_-$, the closure being taken with respect to the graph norm. Indeed, assume this dimension was strictly less than $m$. The orthogonality of the projection also means that our assumption would imply that the closure in graph norm of $\spa\{{P}_k\}_{k=0}^{\infty}$ is a proper subspace of $\cD_{\bf L}$. Lemma \ref{t-norms} says the graph norm and the norm in the corresponding left-definite space are equivalent. However, Lemma \ref{t-leftdefortho} states the closure in the norm induced in the second left-definite space associated with the pair $(\cH,{\bf A}) = (L^2[(a,b),w], {\bf L})$ is $\cD_{\bf L}$. This is a contradiction, so the $\{\widetilde{P}_k\}_{k=0}^{\infty}$ span an $m$ dimensional subspace of $\cD_+\oplus\ci{{\bf  A}} \cD_-$.

This means that the problem is now finite dimensional! In particular, the closure of spans is obvious. Also, there are $m$ functions $\{\widetilde{P}_{k_i}\}_{i=1}^m$ which can be completed to a basis $\{\widetilde P_{k_1},\dots,\widetilde P_{k_m}, h_1, \hdots, h_m\}$ of $\cD_+\oplus\ci{{\bf  A}} \cD_-$. Therefore, the functions $\{\widetilde{P}_{k_i}\}_{i=1}^{m}$ are linear independent modulo $\cD\ti{min}$.

The definition of $ \widetilde{P}_{k}$ implies $P_{k} - \widetilde{P}_{k} \in \cD\ti{min}$. Hence, when viewed in the quotient space $\cD\ti{max}\ominus \cD\ti{min}$, the projection $\widetilde{P}_{k}$ belongs to the same equivalence class as the corresponding eigenfunction ${P}_{k}$, $[\widetilde{P}_{k}]=[P_{k}]$. Invoking the definition of linear independence modulo $\cD\ti{min}$ again, the eigenfunctions $\{P_{k_i}\}_{i=1}^m$ are linearly independent modulo the minimal domain.

Therefore, there are $m$ members of $\{P_{k_i}\}_{i=1}^m$ which define a self-adjoint extension $\widetilde{{\bf L}}$ of ${\bf L}\ti{min}$ in the above fashion, via the GKN1 Theorem (Theorem \ref{t-gkn1}). It remains to verify that 
$\widetilde{{\bf L}} = {\bf L}$. From equation \eqref{e-orthoeigens} we immediately conclude that all of the $\{P_k\}_{k=0}^{\infty}$ belong to the domain $\cD_{\widetilde{{\bf L}}}$. Thus ${\bf L}\subseteq \widetilde{{\bf L}}\subseteq \widetilde{{\bf L}}^*\subseteq {\bf L}^*$, and it is known that ${\bf L}={\bf L}^*$.
\end{proof}

This theorem relied heavily on the GKN1 Theorem and began with a self-adjoint operator generated by left-definite theory to define the boundary conditions imposed to create the left-definite space. However, the GKN theory goes both ways to create a complete framework of both necessary and sufficient conditions. The similar conditions of the GKN2 Theorem allow another statement to be made.

\begin{theo}
Let ${\bf L}$ be a self-adjoint operator that is bounded below on $L^2[(a,b),w]$ with domain $\cD_{\bf L}$ that includes a complete orthogonal system of eigenfunctions. Enumerate the eigenfunctions as $\{P_k\}_{k=0}^{\infty}$. Furthermore, let ${\bf L}$ be an extension of the minimal (symmetric, closed) operator ${\bf L}\ti{min}$ with domain $\cD\ti{min}$, and ${\bf L}\ti{min}$ have deficiency indices $(m,m)$. The $2n$th (for $n\in\NN$) left-definite space generated by the operator ${\bf L}$ will have deficiency indices $(nm,nm)$.

Then, the GKN conditions for the $2n$th left-definite space are given by some $\{P_{k_1},\dots,P_{k_{nm}}\}$.
\end{theo}

\begin{proof}
The $2n$th left-definite space exists and is unique for the operator ${\bf L}$. It also possesses the complete set of orthogonal eigenfunctions that are in the domain of ${\bf L}$, and their spectra coincide. Glazman symmetry conditions and the linear independence of $\{P_{k_1},\dots,P_{k_{nm}}\}$ modulo the minimal domain follow by the same argument as the previous theorem. Indices matching the $2n$th left-definite space with the $n$th power composition of the operator follows from Corollary \ref{t-comppower}. The result follows by the GKN2 Theorem.
\end{proof}

This converse implication allows left-definite spaces to be generated when the operator ${\bf L}$ contains a complete orthogonal system of eigenfunctions. This provides a complete answer to the question of which functions should be considered as GKN conditions to create left-definite spaces. The only further improvement would be to explicitly say which eigenfunctions were sufficient for this purpose. The next section seeks to answer this question when the operator stems from a Sturm--Liouville differential expression.

\section{Ramifications and Conjecture}\label{s-conj}
Let ${\bf L}^n$ be a self-adjoint operator defined by left-definite theory on $L^2[(a,b),w]$ with domain $\cD_{\bf L}^n$ that includes a complete system of orthogonal eigenfunctions. Let ${\bf L}^n$ operate on its domain via $\ell^n[\fdot]$, a differential operator of order $2n$, where $n\in\NN$, generated by composing a Sturm--Liouville differential operator with itself $n$ times. Furthermore, let $\ell^n$ be an extension of the minimal operator ${\bf L}^n\ti{min}$, which has deficiency indices $(m,m)$.

As we are working with coupled boundary conditions for the $n$th power of an operator, it suffices to consider $m=n$. Let us explain why. The case where $m=n$ corresponds to both endpoints of $\ell$ being limit circle. If only one endpoint is limit point, then one GKN condition is still necessary for $\ell$. This GKN condition then imposes a restriction at the limit circle endpoint, while it is just satisfied trivially on the limit point side by all functions in the maximal domain. If both endpoints are limit point, then no boundary conditions are needed. In that case ${\bf L}\ti{min}$ is essentially self-adjoint. Therefore, we assume without loss of generality that the deficiency indices are $(n,n)$ in this section. A more in depth discussion can be found in \cite{BDG, GZ, N, WW}. 

Enumerate the orthogonal eigenfunctions as $\{P_k\}_{k=0}^{\infty}$. Define the following domains:
\begin{align*} 
\cA_n&:=\left\{f:(a,b)\to \CC~\Big|~f,f',\dots,f^{(2n-1)}\in AC\ti{loc}(a,b);
(p(x))^n f^{(2n)}\in L^2[(a,b),w]\right\},
\\
\cB_n&:=\left\{f\in \cD\ti{max}^n~\Big|~
[f,P_j]_n\Big|_a^b=0 \text{ for }j=0,1,\dots,n-1\right\}, 
\\
\cC_n&:=\left\{f\in \cD\ti{max}^n~\Big|~
[f,P_j]_n\Big|_a^b=0 \text{ for any }n \text{ distinct }j\in\NN \right\}, \text{ and}
\\
\cF_n&:=\left\{f\in \cD\ti{max}^n~\Big|~(a_j(x)
y^{(j)}(x))^{(j-1)}\Big|_a^b=0 \text{ for }j=1,2,\dots,n
\right\}.
\end{align*}

The $p(x)$ above is from the standard definition of a Sturm-Liouville differential operator, given in equation \eqref{d-sturmop}, and the $a_j(x)$'s are from the Lagrangian symmetric form of the operator in \eqref{e-lagrangian}.
These domains seem very different, yet progress has already been made in this paper and elsewhere on the equality of these domains. The $n$th left-definite domain, $\cD_{\bf L}^n$ is found to be equal to $\cA_n$ for the Legendre differential operator in \cite[Section 7.5]{LWOG}. A general form for $\cD_{\bf L}^n$ is missing from the literature, so we will assume it is equal to $\cA_n$ for the rest of this section. Indeed, the main condition of $\cA_n$ simply involves the term associated with $f^{(2n)}$ when $\ell^n[f]$ is decomposed into a sum of derivatives of $f$.
 There is a proof of $\cA_n\subseteq \cF_n$ in \cite{LWOG} for the special case where the differential operator $\ell^n[\fdot]$ denotes the $n^{th}$ composite power of the Legendre differential expression and the eigenfunctions $\{P_k\}_{k=0}^{\infty}$ are the Legendre polynomials. This scenario for small $n$ was discussed in Section \ref{s-legendre}. The conditions in $\cF_n$ are particularly significant as they represent easily testable conditions that are not in the GKN format. 

In Section \ref{s-results} it was shown that by a proper re-enumeration of the eigenfunctions $\{P_k\}_{k=0}^{\infty}$, we have $\cA_n=\cB_n$. There we also proved $\cA_n=\cB_n=\cC_n$ for n=2. However, a proof of the general $n$ case for $\cA_n=\cB_n=\cC_n$ is elusive at this stage. It is obvious that $\cB_n\subset \cC_n$.

\begin{conj}\label{c-main}
Let ${\bf L}^n$ be a self-adjoint operator defined by left-definite theory on $L^2[(a,b),w]$ with domain $\cD_{\bf L}^n$ that includes a complete system of orthogonal polynomial eigenfunctions. Let ${\bf L}^n$ operate on its domain via $\ell^n[\fdot]$, a differential operator of order $2n$, where $n\in\NN$, generated by composing a Sturm--Liouville differential operator with itself $n$ times. Furthermore, let ${\bf L}^n$ be an extension of the minimal operator ${\bf L}^n\ti{min}$, which has deficiency indices $(n,n)$. Then $\cA_n=\cB_n=\cC_n=\cF_n=\cD_{\bf L}^n$ $\forall n\in\NN$.
\end{conj}

This conjecture extends one made in \cite[Chapter 9]{LWOG} by the broad conditions in $\cC_n$. Here we prove some subcases.

\begin{theo}\label{t-bd}
Under the hypotheses of Conjecture \ref{c-main}, and the assumption that $\cA_n=\cB_n$,
we have $\cB_n\subseteq \cF_n$ $\forall n\in\NN$.
\end{theo}

The proof utilizes an explicit form of the sesquilinear form, as opposed to the one given in equation \eqref{e-greens}. This representation allows for more precision in defining which limits are disappearing and which are remaining as we approach the endpoints of $(a,b)$. Specifically, for the Sturm--Liouville operators of interest the expression $q(x)/w(x)$ is a constant, and we have
\begin{align} \label{e-sesqui}
[f,g]_n(x)=\sum_{k=1}^n\sum_{j=1}^k (-1)^{k+j}\Big\{(a_k(x)&\overline{g}^{(k)}(x))^{(k-j)}f^{(j-1)
}(x)\\
&-(a_k(x)f^{(k)}(x)^{(k-j)})\overline{g}^{(j-1)}(x)\Big\},\nonumber
\end{align}
where the $a_k(x)$'s are again from the Lagrangian symmetric form of the differential operator \eqref{e-lagrangian}. 
For a reference see e.g.~\cite[Section 10]{LW15}.

\begin{proof}[Proof of Theorem \ref{t-bd}]
We proceed by induction on $n$. The base case is proven by a simple application of Green's Formula for the sesquilinear form. Let $f\in \cB_1$ and recall that in the Strum--Liouville systems of interest we have $q\equiv 0$. We compute
\begin{align*}
0=[f,P_0]_1|_a^b=[f,1]_1|_a^b&=\langle \ell[f],1\rangle_{L^2[(a,b),w]}-\langle f,\ell[1]\rangle_{L^2[(a,b),w]} \\
&=\int_a^b\left(\dfrac{1}{w(x)}[p(x)f'(x)]'\right)w(x)dx-0 \\
&=\lim_{x\to b^-}(p(x)f'(x))-\lim_{x\to a^+}(p(x)f'(x)).
\end{align*}

Assume that $\cB_{n-1}\subseteq \cF_{n-1}$ as the inductive hypothesis. Corollary \ref{t-comppower} shows that 
$$\cD_{\bf L}^n=\cV_{2n}=\cA_n=\cB_n\subset \cB_{n-1}\subseteq \cF_{n-1}.$$
The limits in the description of $\cB_{n-1}$ exist, and are finite, by Theorem \ref{t-limits}. The inclusion $\cB_n\subset \cB_{n-1}$ can be shown using Green's formula. Consequently, $f\in \cB_n$ implies that $\lim_{x\to b^-}(a_i(x)f^{(i)}(x))^{(i-1)}-\lim_{x\to a^+}(a_i(x)f^{(i)}(x))^{(i-1)}=0$, $\forall i=1,\dots,n-1$. 

Furthermore, the definition of $\cB_n$ includes the condition $[f,P_0]_n|_a^b=[f,1]_n|_a^b=0$. Hence, the only nonzero terms in equation \eqref{e-sesqui} are when $j=1$, yielding $n$ terms. The first $n-1$ of these terms are precisely those given in the definition of $\cF_{n-1}$, explicitly:
\begin{align*}
0=[f,1]_n|_a^b&=\lim_{x\to b^-}[-a_1f'+(a_2f'')'-\dots+(-1)^n(a_nf^{(n)})^{(n-1)}] \\
&\quad\,-\lim_{x\to a^+}[-a_1f'+(a_2f'')'-\dots+(-1)^n(a_nf^{(n)})^{(n-1)}] \\
&=\lim_{x\to b^-}(a_nf^{(n)})^{(n-1)}-\lim_{x\to a^+}(a_nf^{(n)})^{(n-1)}.
\end{align*}
Hence also the last condition in $\cF_n$ is satisfied, and thus $f\in \cF_n$.
The claim that $\cB_n\subseteq \cF_n$ follows by induction on $n$. 
\end{proof}

For the reverse inclusion, the proof will involve working with the sesquilinear form explicitly. Hence, the differences between Sturm--Liouville operators arise primarily in the definition of the $a_k(x)$'s in equation \eqref{e-sesqui}. The following theorem is formulated for the Legendre operator.
%



\begin{theo}
Let the hypotheses of Conjecture \ref{c-main} hold, where $\ell^n$ is the classical Legendre differential expression given in \eqref{e-legendre} Assume that for all $f\in \cF_n$, $f'',\dots,f^{(2n-2)}\in L^2[(a,b),dx]$. Then we have $\cF_n\subseteq \cC_n\subseteq \cB_n$ $\forall n\in\NN$.
\end{theo}

\begin{proof}
Fix $n\in\NN$. The differential expression $\ell^n$ can be written in Lagrangian symmetric form as 
\begin{align}
\ell^n[f](x)=\sum_{k=1}^n(-1)^k{n \brace k}_2((1-x^2)^k f^{(k)}(x))^{(k)},
\end{align}
where ${n \brace k}_2$ denote the Legendre--Stirling numbers of the second kind, see \cite{AGL} for more. Hence, $a_k(x)=C(n,k)(1-x^2)^k$, where $C(n,k)$ is a constant.
Deconstruct the explicit expression for $[f,P_s]_n|_a^b$ into the following:
\begin{align*}
LHS&=\sum_{k=1}^n\sum_{j=1}^k(-1)^{k+j}[a_k(x)P_s^{(k)}(x)]^{(k-j)}f^{(j-1)}(x), \\
RHS&=\sum_{k=1}^n\sum_{j=1}^k(-1)^{k+j+1}[a_k(x)f^{(k)}(x)]^{(k-j)}P_s^{(j-1)}(x).
\end{align*}
Assume that $s\geq n$ so that all terms are nonzero. The case where $s<n$ will immediately follow.
The assumption that $f\in \cF_n$ means that
\begin{align*}
\lim_{x\to b^-}[a_k(x)f^{(k)}(x)]^{(k-1)}-\lim_{x\to a^+}[a_k(x)f^{(k)}(x)]^{(k-1)}=0,
\end{align*}
for $k=1,\dots,n$, so terms of this form in the RHS will not be of concern. At a single endpoint, consider limits of the form
\begin{align*}
\lim_{x\to a^+(\text{or }b^-)}[a_k(x)f^{(k)}(x)]^{(k-2)},
\end{align*}
for $k=2,\dots,n$. Assume $c_2\neq 0$. Without loss of generality, assume that at the endpoint $b$ the above limit is equal to $c_2>0$ and is finite. Define $r_2:=c_2/2$. Then
\begin{align*}
r_2<\lim_{x\to b^-}[a_k(x)f^{(k)}(x)]^{(k-2)}=\lim_{x\to b^-}\left(\sum_{i=0}^{k-2}\binom{k-2}{i}a_k^{(k-2-i)}(x)f^{(k+i)}(x)\right).
\end{align*}
Recall that $a_k(x)=C(n,k)(1-x^2)^k$, so each term on the right hand side will possess a factor of $(1-x^2)^2$ after differentiation. Dividing both sides by this factor yields
\begin{align}\label{e-divide}
\lim_{x\to b^-}\dfrac{r_2}{(1-x^2)^2}<\lim_{x\to b^-}\left(\sum_{i=0}^{k-2}\binom{k-2}{i}\widetilde{a}_k^{(k-2-i)}(x)f^{(k+i)}(x)\right)=:S_2,
\end{align}
where we use $$\widetilde{a}_k^{(k-2-i)}(x)=\dfrac{a_k^{(k-2-i)}(x)}{(1-x^2)^2}.$$
Note $k\geq 2$ necessarily here, so the relevant derivatives of $f(x)$ in $S_2$ are $f''(x),\dots,f^{(n-2)}(x)$, which are all in $L^2[(a,b),dx]$ by assumption. However, $(1-x^2)^k\in L^2[(a,b),dx]$ for all $k\in\NN_0$, and so are its derivatives because they are all polynomials. Hence, each summand in $S_2$ is in $L^1[(a,b),dx]$ by the Cauchy--Schwarz inequality, and the finite sum $S_2$ is then also in $L^1[(a,b),dx]$. It is apparent that
$$\lim_{x\to b^-}\dfrac{r_2}{(1-x^2)^2}=\infty.$$
The comparison test thus yields a contradiction to the fact that $S_2\in L^1[(a,b),dx]$. As $c_2\neq 0$ was arbitrary, we conclude
\begin{align*}
\lim_{x\to b^-}[a_k(x)f^{(k)}(x)]^{(k-2)}=0.
\end{align*}
A similar argument shows that the same result at the endpoint $a$. The method outlined above can be used to show \emph{mutatis mutandis} that
\begin{align*}
\lim_{x\to b^-}[a_k(x)f^{(k)}]^{(k-j)}=0
\end{align*}
for $2< j\leq n$. 
The eigenfunctions $P_s(x)$ are assumed to be polynomials so we may write $P_s(x)=\sum_{h=0}^s\alpha_h x^h$. Then the RHS can be broken down into a power of $x$ times the above limit for each value of $k,j<m$. Basic limit laws say that, for $h\in\NN$, 
\begin{align*}
\lim_{x\to b^-}[a_k(x)f^{(k)}]^{(k-j)}P_s(x)=\sum_{h=0}^s\left(\lim_{x\to b^-}\alpha_h x^h [a_k(x)f^{(k)}]^{(k-j)}\right)=0,
\end{align*}
by splitting up the product and using the fact that $a$ and $b$ are finite endpoints. This means
\begin{align*}
RHS&=\sum_{k=1}^n\sum_{j=1}^k(-1)^{k+j+1}[a_k(x)f^{(k)}(x)]^{(k-j)}P_s^{(j-1)}(x)=0
\end{align*}
for any $s\in\NN$. Likewise, 
\begin{align*}
LHS&=\sum_{k=1}^n\sum_{j=1}^k(-1)^{k+j}[a_k(x)P_s^{(k)}(x)]^{(k-j)}f^{(j-1)}(x)=0.
\end{align*}
The theorem now follows, as this collectively shows
$[f,P_s]_n|_a^b=0, \text{ for }s\in\NN.$
\end{proof}

The method of proof developed above applies to more than just the Legendre differential expression, with minimal alterations.

\begin{theo}
Let the hypotheses of Conjecture \ref{c-main} hold, where $\ell^n$ is a classical Jacobi or Laguerre differential expression with $\al,\beta>-1$ or $\al>-1$ respectively. Assume that for all $f\in \cF_n$, $f'',\dots,f^{(2n-2)}\in L^2[(a,b),dx]$. Then we have $\cF_n\subseteq \cC_n\subseteq \cB_n$ $\forall n\in\NN$.
\end{theo}

\begin{proof}
Fix $n\in\NN$. The Jacobi differential expression $\ell_{\bf J}^n$ can be written in Lagrangian symmetric form as 
\begin{align*}
\ell_{\bf J}^n[f](x)=\dfrac{1}{(1-x)^{\al}(1+x)^{\beta}}\sum_{k=1}^n(-1)^k(C(n,k,\al,\beta)(1-x)^{\al+k}(1+x)^{\beta+k}f^{(k)}(x))^{(k)},
\end{align*}
where $C(n,k,\al,\beta)$ is a constant, and $\al,\beta>-1$. The above proof carries through the same as above, as the term divided through in analogy to equation \eqref{e-divide} will be $(1-x)^{\al+2}(1+x)^{\beta+2}$, and has adequate blow-up at the endpoints.

The Laguerre differential expression $\ell_{\bf L}^n$ can be written in Lagrangian symmetric form as 
\begin{align*}
\ell_{\bf L}^n[f](x)=\dfrac{1}{x^{\al}e^{-x}}\sum_{k=1}^n(-1)^k(C(n,k,\al)x^{\al+k}e^{-x}f^{(k)}(x))^{(k)},
\end{align*}
where $C(n,k,\al)$ is a constant, and $\al>-1$. The term divided through in analogy to equation \eqref{e-divide} will be $x^{\al+2}$, and has adequate blow-up at $0$. 

More details concerning the setup and properties of these differential equations can be found in \cite{D,ELT}.
\end{proof}

This effectively covers most of the classical differential equations which possess complete sets of orthogonal polynomial eigenfunctions. The Hermite equation was not discussed because it does not require boundary conditions of the above forms, as it is limit-point at both $-\infty$ and $\infty$.

The astute reader may have noticed that the explicit conditions of $\cF_n$ did not play a large part in the above proofs.
Indeed, the limit conditions in the definition of $\cF_n$ are necessary to prove the assumption that $f'',\dots,f^{(2n-2)}\in L^2[(a,b),dx]$, which was essential. This implication can be accomplished using an appropriate choice of the vectors $\psi$ and $\varphi$ in two applications of the so-called CHEL Theorem {\cite[Theorem 8.7]{LWOG}}, and subtracting them to form coupled boundary conditions. The limit conditions of $\cF_n$ arise as boundary terms from integrating $\ell^n$ in Lagrangian symmetric form, and can be used to show certain functions are in $L^2[(a,b),dx]$.

It is also important to note that the specific choice of limits in $\cF_n$ cannot be altered. Together, they ensure that $[f,P_0]_n|_a^b=[f,1]_n|_a^b=0$. The function $1$ also happens to be included in every classical orthogonal polynomial sequence so it is particularly applicable. 

Finally, with these few extra assumptions, Conjecture \ref{c-main} has been shown.

\begin{theo}
Assume the hypotheses of Conjecture \ref{c-main}, where $\ell^n$ is a classical Jacobi or Laguerre differential expression, with $\al,\beta>-1$ or $\al>-1$ respectively.
Assume $\cA_n=\cB_n$ and that $f\in \cF_n$ implies that $f'',\dots,f^{(2n-2)}\in L^2[(a,b),dx]$. Then $\cA_n=\cB_n=\cC_n=\cF_n=\cD_{\bf L}^n$ $\forall n\in\NN$.
\end{theo}


\end{document}